\documentclass[12pt,oneside]{amsart}

\usepackage{geometry}                
\usepackage{graphicx}
\usepackage{amssymb}
\usepackage{color}
\usepackage{bbm, dsfont}
\usepackage[hidelinks]{hyperref}

\hypersetup{
	colorlinks=false,
	pdfborder={0 0 0},
	pdfborderstyle={/S/U/W 0},
}

\numberwithin{equation}{section}

\usepackage{verbatim}

\usepackage[dvipsnames]{xcolor}

\usepackage{ulem}
\newtheorem{theorem}{Theorem}[section]
\newtheorem{lemma}[theorem]{Lemma}
\newtheorem{corollary}[theorem]{Corollary}

\theoremstyle{definition}

\newtheorem{conjecture}[theorem]{Conjecture}

\theoremstyle{remark}
\newtheorem{remark}[theorem]{Remark}
\newtheorem*{remark*}{Note}

\numberwithin{equation}{section}

\usepackage{amsmath}
\usepackage{mathrsfs}

\usepackage{verbatim}
\usepackage{amsmath,amssymb,amsthm}           
\usepackage{amssymb}            
\usepackage{amsfonts}            
\usepackage{mathrsfs}          
\usepackage{amsthm}
\usepackage{algorithm}  
\usepackage{algorithmicx}  
\usepackage{algpseudocode}  
\usepackage{mathtools}
\usepackage{commath}
\usepackage{physics}
\usepackage{bm}
\usepackage{graphicx}

\usepackage{float}
\usepackage{listings}
\usepackage{subfigure}
\usepackage{multirow}
\usepackage{color}
\usepackage{bbm}
\usepackage[export]{adjustbox}
\usepackage{enumerate}
\usepackage{bbm}

\usepackage{amsmath}
\usepackage{mathrsfs}

\newcommand{\RNum}[1]{\uppercase\expandafter{\romannumeral #1\relax}}

\newcommand{\specificthanks}[1]{\@fnsymbol{#1}}

\DeclareFontFamily{OML}{rsfs}{\skewchar\font'177}
\DeclareFontShape{OML}{rsfs}{m}{n}{ <5> <6> rsfs5 <7> <8> <9>
	rsfs7 <10> <10.95> <12> <14.4> <17.28> <20.74> <24.88> rsfs10 }{}
\DeclareMathAlphabet{\mathfs}{OML}{rsfs}{m}{n}

\newcounter{cnstcnt}

\newcommand{\cref}[1]{\ensuremath{c_{\ref*{#1}}}}

\newcounter{newcnstcnt}

\newcommand{\Cref}[1]{\ensuremath{C_{\ref*{#1}}}}

\DeclareFontFamily{U}{mathx}{}
\DeclareFontShape{U}{mathx}{m}{n}{<-> mathx10}{}
\DeclareSymbolFont{mathx}{U}{mathx}{m}{n}
\DeclareMathAccent{\widehat}{0}{mathx}{"70}
\DeclareMathAccent{\widecheck}{0}{mathx}{"71}

\begin{document}

	\title{On the gap between cluster dimensions of loop soups on $\mathbb{R}^3$ and the metric graph of $\mathbb{Z}^3$}




	

		\author{Zhenhao Cai$^1$}
		\address[Zhenhao Cai]{Faculty of Mathematics and Computer Science, Weizmann Institute of Science}
		\email{zhenhao.cai@weizmann.ac.il}
		\thanks{$^1$Faculty of Mathematics and Computer Science, Weizmann Institute of Science}

		\author{Jian Ding$^2$}
		\address[Jian Ding]{New Cornerstone Science Laboratory, School of Mathematical Sciences, Peking University}
		\email{dingjian@math.pku.edu.cn}
		\thanks{$^2$New Cornerstone Science Laboratory, School of Mathematical Sciences, Peking University}
	
	
	\maketitle
	%
	%

	 	\begin{abstract}


     The question of understanding the scaling limit of metric graph critical loop soup clusters and its relation to loop soups in the continuum appears to be one of the subtle cases that reveal interesting new scenarios about scaling limits, with a mixture of macroscopic and microscopic randomness. In the present paper, we show that in three dimensions, scaling limits of the metric graph clusters are strictly larger than the clusters of the limiting continuum Brownian loop soup. We actually show that the upper box-counting dimension of the latter clusters is strictly smaller than $5/2$, while that of the former is $5/2$.


  \end{abstract}

\section{Introduction}\label{section_intro}

In this paper, we study the percolation of loop soups, with a particular focus on the fractal dimension of their clusters. Roughly speaking, a loop soup is a random collection of loops, which are paths on discrete graphs or curves in continuum spaces that start and end at the same point. As reviewed below, this model possesses elegant geometric properties and exhibits profound connections to numerous models in statistical physics. We focus on three types of spaces: the integer lattice $\mathbb{Z}^d$, the Euclidean space $\mathbb{R}^d$, and the metric graph $\widetilde{\mathbb{Z}}^d$, where $d$ represents the dimension (unless otherwise specified, we assume that $d\ge 3$). The metric graph $\widetilde{\mathbb{Z}}^d$ (a.k.a. cable graph or cable system) is constructed as follows. Let $\mathbb{L}^d:=\{\{x,y\}:x,y\in \mathbb{Z}^d,\|x-y\|=1\}$ denote the edge set of $\mathbb{Z}^d$ (where $\|\cdot\|$ represents the Euclidean norm). For each edge $e=\{x,y\}\in \mathbb{L}^d$, we assign a compact interval $I_e$ of length $d$ with two endpoints identical to $x$ and $y$ respectively. The metric graph $\widetilde{\mathbb{Z}}^d$ is then defined as the union of all these intervals. Loop soups on $\Omega\in \{\mathbb{Z}^d,\mathbb{R}^d, \widetilde{\mathbb{Z}}^d\}$ can be defined via a unified procedure. Precisely, we denote by $\{X_t^{\Omega}\}_{t\ge 0}$ the canonical diffusion on $\Omega$: 
\begin{itemize}

	\item  When $\Omega=\mathbb{Z}^d$, $\{X_t^{\mathbb{Z}^d}\}_{t\ge 0}$ is the continuous-time simple random walk on $\mathbb{Z}^d$ with jump rate $1$. I.e., for any $x,y\in \mathbb{Z}^d$ and $t\ge 0$, 
\begin{equation}\label{def_continuous_time_RW}
\mathbb{P}\big( X^{\mathbb{Z}^d}_{t+\Delta t}=y  \mid X^{\mathbb{Z}^d}_t=x \big) =  \left\{
\begin{aligned}
 &(2d)^{-1}\Delta t \cdot \mathbbm{1}_{\{x,y\}\in \mathbb{L}^d} + o(\Delta t)   & \text{if}\ x\neq y; \\
&1-\Delta t + o(\Delta t)  & \text{if}\ x= y.
\end{aligned}
\right.
\end{equation}	 	

	\item  When $\Omega=\mathbb{R}^d$, $\{X_t^{\mathbb{R}^d}\}_{t\ge 0}$ is the standard Brownian motion on $\mathbb{R}^d$. Note that the projection of $X_{\cdot}^{\mathbb{R}^d}$ to each coordinate is a standard one-dimensional Brownian motion, which has variance $1$ at time $1$.

	\item  When $\Omega=\widetilde{\mathbb{Z}}^d$, within each interval $I_e$, $\{X_t^{\widetilde{\mathbb{Z}}^d}\}_{t\ge 0}$ behaves as a standard one-dimensional Brownian motion. When reaching a lattice point $x\in \mathbb{Z}^d$, it evolves as a Brownian excursion along the interval uniformly selected from $\{I_{\{x,y\}}\}_{\{x,y\}\in \mathbb{L}^d}$. After hitting a lattice point (either returning to $x$ or reaching a neighbor of $x$), it resumes in the same manner.

\end{itemize}
For any $v_1,v_2\in \Omega$ and $t\ge 0$, let $q_t^{\Omega}(v_1,v_2)$ denote the transition density of $X_\cdot^{\Omega}$ from $v_1$ to $v_2$ at time $t$, and let $\mathbb{P}_{v_1,v_2,t}^{\Omega}$ denote the law of the bridge of $X_\cdot^{\Omega}$ from $v_1$ to $v_2$ with duration $t$ (whose transition density is given by $\frac{q_s^{\Omega}(v_1,\cdot )q_{t-s}^{\Omega}(\cdot ,v_2)}{q_t^{\Omega}(v_1,v_2)}$ for $0\le s\le t$). The loop measure $\mu^{\Omega}$ is then defined by 
\begin{equation}\label{def_loop_measure}
	\mu^{\Omega}(\cdot):= \int_{v\in \Omega} \mathrm{d}\mathrm{m}^{\Omega}(v)\int_{t> 0} t^{-1}q_t^{\Omega}(v,v)\mathbb{P}_{v,v,t}^{\Omega}(\cdot) \mathrm{d}t.
\end{equation}
Here $\mathrm{m}^{\mathbb{Z}^d}$ is the counting measure on $\mathbb{Z}^d$, and for $\Omega\in \{\mathbb{R}^d, \widetilde{\mathbb{Z}}^d\}$, $\mathrm{m}^{\Omega}$ is the Lebesgue measure on $\Omega$. For each $\alpha>0$, the loop soup on $\Omega$ of intensity $\alpha$, denoted by $\mathcal{L}_\alpha^{\Omega}$, is defined as the Poisson
point process with intensity measure $\alpha \mu^{\Omega}$.

The loop soup $\mathcal{L}_{\alpha}^{\widetilde{\mathbb{Z}}^d}$ was originally introduced in \cite{lupu2016loop} as an extension of $\mathcal{L}_{\alpha}^{\mathbb{Z}^d}$, which stems from the following connection between the diffusions $X_t^{\widetilde{\mathbb{Z}}^d}$ and $X_t^{\mathbb{Z}^d}$. First of all, the projection of $X_t^{\widetilde{\mathbb{Z}}^d}$ onto $\mathbb{Z}^d$ has the same law as $X_t^{\mathbb{Z}^d}$. In addition, for each $\Omega\in \{\mathbb{Z}^d,\widetilde{\mathbb{Z}}^d\}$, from the time when $X_t^{\Omega}$ departs from a lattice point $x\in \mathbb{Z}^d$ until it first hits a neighbor of $x$, the occupation time at $x$ is an independent exponential random variable with mean $1$. Here the occupation time is the holding time of this jump when $\Omega=\mathbb{Z}^d$, and is the total local time at $x$ of all Brownian excursions during this period when $\Omega=\widetilde{\mathbb{Z}}^d$. Consequently, the Green's function, defined by   
\begin{equation}
	G^{\Omega}(v_1,v_2):= \mathbb{E}_{v_1}^{\Omega}\big[ \mathsf{L}_{v_2}(X_\cdot^{\Omega})  \big] = \int_{0}^{\infty} q_t^{\Omega}(v_1,v_2) dt, \ \ \forall v_1,v_2\in \mathbb{Z}^d,
	\end{equation}
is consistent for $\Omega\in \{\mathbb{Z}^d,\widetilde{\mathbb{Z}}^d\}$. Here $\mathbb{E}_{v_1}^{\Omega}$ denotes the expectation under $\mathbb{P}_{v_1}^{\Omega}$ (i.e., the law of $X_\cdot^{\Omega}$ starting from $v_1$), and $\mathsf{L}_{v_2}(\cdot)$ denotes the total occupation time at $v_2$. Moreover, the range (i.e., the set of visited points) of $X_\cdot^{\widetilde{\mathbb{Z}}^d}$ can be recovered by taking the union of all edges traversed by $X_\cdot^{\mathbb{Z}^d}$ and then adding independent Brownian excursions $\mathcal{E}_{x}(\mathsf{L}_{x}(X_\cdot^{\mathbb{Z}^d}))$ for all $x\in \mathbb{Z}^d$. Here $\mathcal{E}_{x}(t)$ represents Brownian excursions on $\widetilde{\mathbb{Z}}^d$ from $x$ that have total local time $t$ at $x$ and are conditioned to return to $x$ before reaching any neighbor of $x$. As noted in \cite[Section 2]{lupu2016loop}, this correspondence naturally leads to a coupling between the loops in $\mathcal{L}_\alpha^{\mathbb{Z}^d}$ and $\mathcal{L}_\alpha^{\widetilde{\mathbb{Z}}^d}$, which satisfies the following properties: 
 \begin{enumerate}

 	\item[(i)]  There is a bijection $\varpi(\cdot)$ between the loops in $\mathcal{L}_\alpha^{\mathbb{Z}^d}$ and $\mathcal{L}_\alpha^{\widetilde{\mathbb{Z}}^d}$ that visit at least two lattice points. Moreover, given any of these loops $\ell$ in $\mathcal{L}_\alpha^{\mathbb{Z}^d}$, the range of $\varpi(\ell)$ can be obtained from that of $\ell$ by adding Brownian excursions $\mathcal{E}_{x}(\mathsf{L}_{x}(\ell))$ for all $x\in \mathbb{Z}^d$.

 	\item[(ii)]  In $\mathcal{L}_\alpha^{\mathbb{Z}^d}$ there are loops that stay at a single lattice point, each carrying a positive holding time at that point. For each $x\in \mathbb{Z}^d$, given these loops (assume that their total holding time at $x$ is $L_x$), the union of the ranges of loops in $\mathcal{L}_\alpha^{\widetilde{\mathbb{Z}}^d}$ that visit $x$ but no other lattice points has the same distribution as $\mathcal{E}_{x}(L_x)$; in addition, these unions are independent across $x\in \mathbb{Z}^d$.


 	

 	\item[(iii)] Every loop in $\mathcal{L}_\alpha^{\mathbb{Z}^d}$ is covered by either Property (i) or (ii). However, the loops in $\mathcal{L}_\alpha^{\widetilde{\mathbb{Z}}^d}$ that do not visit any lattice point (i.e., the whole loop is contained in the interior of a single interval) have not been accounted for. For each $e\in \mathbb{L}^d$, the loops in $\mathcal{L}_\alpha^{\widetilde{\mathbb{Z}}^d}$ contained in $I_e$ are independent of both $\mathcal{L}_\alpha^{\mathbb{Z}^d}$ and the loops contained in any other interval.

 \end{enumerate}
As in \cite{cai2024high}, we refer to the loops described in Properties (i) and (ii) as fundamental loops and point loops respectively. The loops of $\mathcal{L}_\alpha^{\widetilde{\mathbb{Z}}^d}$ described in Property (iii) are called edge loops.

 At intensity $\alpha=\frac{1}{2}$, the loop soup has a profound connection with the Gaussian free field (GFF). This connection can be characterized via the isomorphism theorem, as described below. For $\Omega\in \{\mathbb{Z}^d,\widetilde{\mathbb{Z}}^d\}$, the GFF on $\Omega$, denoted by $\{ \phi_v^{\Omega} \}_{v\in \Omega}$, is a mean-zero Gaussian field with covariance 
 \begin{equation}\label{def_GFF}
 	\mathbb{E}\big[ \phi_{v_1}^{\Omega}\phi_{v_2}^{\Omega} \big] = G^{\Omega}(v_1,v_2), \ \ \forall v_1,v_2\in \Omega. 
 \end{equation}
 By the consistency of $G^{\mathbb{Z}^d}(\cdot,\cdot)$ and $G^{\widetilde{\mathbb{Z}}^d}(\cdot,\cdot)$ restricted to $\mathbb{Z}^d$, one has 
 \begin{equation}
 	\big\{ \phi_v^{\mathbb{Z}^d} : v\in \mathbb{Z}^d  \big\} \overset{\mathrm{d}}{=}	\big\{ \phi_v^{\widetilde{\mathbb{Z}}^d} : v\in \mathbb{Z}^d \big\}. 
 \end{equation}
 In addition, conditioned on the values of $\phi_\cdot^{\widetilde{\mathbb{Z}}^d}$ on the lattice $\mathbb{Z}^d$, for each $\{x,y\}\in \mathbb{L}^d$, $\{\phi_v^{\widetilde{\mathbb{Z}}^d}\}_{v\in I_{\{x,y\}}}$ is an independent Brownian bridge with boundary conditions $\phi_x^{\widetilde{\mathbb{Z}}^d}$ and $\phi_y^{\widetilde{\mathbb{Z}}^d}$, generated by Brownian motions with variance $2$ at time $1$. The isomorphism theorem, which was first presented in \cite{le2011markov} for discrete graphs and later extended in \cite{lupu2016loop} to metric graphs, states that for $\Omega\in \{\mathbb{Z}^d,\widetilde{\mathbb{Z}}^d\}$, there exists a coupling between the GFF $\{ \phi_v^{\Omega} \}_{v\in \Omega}$ and the loop soup $\mathcal{L}_{1/2}^{\Omega}$ such that  
 \begin{equation}\label{indentity_iso}
 	\tfrac{1}{2}\big(\phi_v^{\Omega}\big)^2=\widehat{\mathcal{L}}_{1/2}^{\Omega,v}, \ \ \forall v\in \Omega. 
 \end{equation}
 Here $\widehat{\mathcal{L}}_{1/2}^{\Omega,v}$ denotes the total occupation time at $v$ of all loops in $\mathcal{L}_{1/2}^{\Omega}$. See also \cite{10.1214/ECP.v17-1792} for a generalized version involving random interlacements, and \cite{drewitz2022cluster} for thorough discussions on general graphs. As reviewed below, the isomorphism theorem builds an essential bridge between loop soups and GFFs, and significantly propels their development, particularly from the perspective of percolation.



  The percolation of clusters induced by the loop soup $\mathcal{L}_{\alpha}^{\mathbb{Z}^d}$ was first studied in \cite{chang2016phase}. Here each edge is considered open if and only if it is traversed by at least one fundamental loop. In this percolation model, the states of different edges are correlated, unlike in Bernoulli percolation. It was shown in \cite{chang2016phase} that as the intensity $\alpha$ varies, $\mathcal{L}_{\alpha}^{\mathbb{Z}^d}$ exhibits a non-trivial phase transition in percolation. Various bounds on connecting probabilities were obtained in \cite{chang2016phase, chang2017supercritical}. At this stage, almost all estimates merely rely on combinatorial considerations, and are usually not sharp. The advent of the metric graph loop soup in \cite{lupu2016loop} marked a shift to the stage of rigorous quantitative estimation in this field, as we sketch below. Under the coupling in the isomorphism theorem, the identity (\ref{indentity_iso}) implies that every loop cluster of $\mathcal{L}_{1/2}^{\widetilde{\mathbb{Z}}^d}$ is exactly a sign cluster of $\phi_\cdot^{\widetilde{\mathbb{Z}}^d}$. Here a ``loop cluster'' means a maximal connected subgraph of $\widetilde{\mathbb{Z}}^d$ consisting of loops in $\mathcal{L}_{1/2}^{\widetilde{\mathbb{Z}}^d}$, and a ``sign cluster'' means a maximal connected subgraph where the signs of $\phi_\cdot^{\widetilde{\mathbb{Z}}^d}$ are constant. Moreover, since $\phi_\cdot^{\widetilde{\mathbb{Z}}^d}$ is continuous, its values on the boundary of each sign cluster must be zero (the same applies to the occupation field $\widehat{\mathcal{L}}_{1/2}^{\widetilde{\mathbb{Z}}^d,\cdot}$). As shown in \cite[Proposition 5.2]{lupu2016loop}, this observation together with the symmetry and correlation (\ref{def_GFF}) of $\phi_\cdot^{\widetilde{\mathbb{Z}}^d}$ leads to the following explicit formula for the two-point function of $\mathcal{L}_{1/2}^{\widetilde{\mathbb{Z}}^d}$: 
  \begin{equation}\label{two_point_function}
  	\mathbb{P}\big(v_1\xleftrightarrow{\cup \mathcal{L}_{1/2}^{\widetilde{\mathbb{Z}}^d}} v_2\big)=  \frac{2}{\pi}\arcsin\Big( \tfrac{G^{ \widetilde{\mathbb{Z}}^d}(v_1,v_2)}{\sqrt{G^{ \widetilde{\mathbb{Z}}^d}(v_1,v_1)G^{ \widetilde{\mathbb{Z}}^d}(v_2,v_2)  }} \Big),   \ \forall v_1,v_2\in \widetilde{\mathbb{Z}}^d. 
  \end{equation}
   Here we use $\cup \mathcal{L}$ to denote the union of ranges of all loops in the point measure $\mathcal{L}$, and we write $A\xleftrightarrow{F} A'$ for the event that there exists a path within $F$ connecting $A$ and $A'$. Referring to \cite[Sections 5 and 6]{lupu2016loop}, the following two fundamental consequences can be derived from the isomorphism theorem and (\ref{two_point_function}): 
 \begin{enumerate}
  

 	\item[(1)]  For any $\alpha\in (0,\frac{1}{2}]$, the loop soup $\mathcal{L}_{\alpha}^{\widetilde{\mathbb{Z}}^d}$ does not percolate;

 	\item[(2)]   For the GFF level-set $\widetilde{E}^{\ge h}:=\{v\in \widetilde{\mathbb{Z}}^d: \phi_v^{\widetilde{\mathbb{Z}}^d}\ge h \}$ ($h\in \mathbb{R}$), the critical level (with respect to $h$) for percolation exactly equals zero.

 \end{enumerate}
   Notably, it was shown in \cite{chang2024percolation} that $\mathcal{L}_{\alpha}^{\widetilde{\mathbb{Z}}^d}$ percolates for all $\alpha>\frac{1}{2}$, which together with (1) verifies that the critical intensity of $\mathcal{L}_{\alpha}^{\widetilde{\mathbb{Z}}^d}$ equals $\frac{1}{2}$.

   A series of works sought to extend the two-point estimate (\ref{two_point_function}) to point-to-set estimates, particularly for the one-arm probability. To be precise, we use the unified notation $\bm{0}$ to denote the origin for all $\Omega\in \{\mathbb{Z}^d,\mathbb{R}^d,\widetilde{\mathbb{Z}}^d\}$. For any $A\subset \mathbb{Z}^d$, we denote its boundary by $\partial A:=\{ x\in A:  \exists y\in A^c\ \text{such that}\ \{x,y\}\in \mathbb{L}^d\}$. In addition, for any $N\ge 1$, we denote the box $B(N):=[-N,N]^d\cap \mathbb{Z}^d$. For any $d\ge 3$, the one-arm probability for $\mathcal{L}_{1/2}^{\widetilde{\mathbb{Z}}^d}$ is defined by 
   \begin{equation}\label{def_one_arm_prob}
   	\theta_d(N):= \mathbb{P}\big( \bm{0}\xleftrightarrow{\cup \mathcal{L}_{1/2}^{\widetilde{\mathbb{Z}}^d}} \partial B(N) \big). 
   \end{equation}
    (\textbf{P.S.} In previous papers \cite{cai2024high, cai2024one, cai2024quasi, cai2024incipient}, we used $\theta_d(\cdot)$ to denote the one-arm probability of the positive level-set of the GFF $\phi_\cdot^{\widetilde{\mathbb{Z}}^d}$. By the isomorphism theorem, this probability is proportional to that in (\ref{def_one_arm_prob}); the same applies to $\rho_d(\cdot, \cdot)$ in (\ref{def_rho_crossing}).) It has been established through a succession of works \cite{ding2020percolation, drewitz2023critical, cai2024high, drewitz2023arm, drewitz2024critical, cai2024one} that 
    \begin{align}
	&\text{when}\ 3\le d<6,\ \ \ \ \ \ \ \ \  \theta_d(N) \asymp N^{-\frac{d}{2}+1};\label{one_arm_low} \\
	&\text{when}\ d=6,\ \ \ \ \  N^{-2}\lesssim \theta_6(N) \lesssim N^{-2+\varsigma(N)}, \  \text{where}\ \varsigma(N):= \tfrac{\ln\ln(N)}{\sqrt{\mathrm{ln}(N)}}\ll 1; \label{one_arm_6} \\
	&\text{when}\ d>6,\ \ \ \ \ \ \ \ \ \ \ \ \ \ \ \theta_d(N) \asymp N^{-2}.\label{one_arm_high}
\end{align}
 Here ``$f\lesssim g$'' means that there exists a constant $C>0$ depending only on $d$ such that $f\le Cg$, and ``$f\asymp g$'' means $f\lesssim g\lesssim f$. At the critical dimension $d_c=6$, though the exponent of $\theta_6(N)$ has been derived in (\ref{one_arm_6}), its exact order remains unknown. It was conjectured in \cite{cai2024one} that $\theta_6(N)\asymp N^{-2}[\mathrm{ln}(N)]^{\delta}$ for some $\delta>0$. As noted in previous works \cite{cai2024quasi, cai2024incipient}, the diverging disparity between the upper and lower bounds in (\ref{one_arm_6}) poses significant complexity in our analysis. In view of this, we omit the six-dimensional case in the following discussion. In addition to $\theta_d(\cdot)$, for any $N> n\ge 1$, the crossing probability 
 \begin{equation}\label{def_rho_crossing}
 	\rho_d(n,N):=\mathbb{P}\big(B(n)\xleftrightarrow{\cup \mathcal{L}_{1/2}^{\widetilde{\mathbb{Z}}^d}}\partial B(N)\big)
 \end{equation}
was also studied in \cite{cai2024one}. Precisely, it was established in \cite[Theorem 1.2]{cai2024one} that 
    \begin{align}
	&\text{when}\ 3\le d<6,\  \rho_d(n,N) \asymp \big(n/N\big)^{\frac{d}{2}-1};\label{crossing_low} \\
	&\text{when}\ d>6, \ \ \ \ \   \ \    \rho_d(n,N) \asymp  (n^{d-4}N^{-2})\land 1. \label{crossing_high}
\end{align}
Beyond these estimates, the geometric structure of loop clusters of $\mathcal{L}_{1/2}^{\widetilde{\mathbb{Z}}^d}$ (equivalently, sign clusters of $\phi_\cdot^{\widetilde{\mathbb{Z}}^d}$) has been studied from several perspectives, including their volume growth \cite{cai2024quasi, drewitz2024cluster} and incipient infinite clusters \cite{ganguly2024ant, cai2024incipient, ganguly2024critical, werner2025switching}. A detailed account can be found in the companion paper \cite{inpreparation_pivotal}.

The main motivation of this paper arises from the following conjecture:

  \begin{conjecture}[{\cite[Conjecture A]{werner2021clusters}}]\label{conj_3d}
	The scaling limit of loop clusters of $\mathcal{L}_{1/2}^{\widetilde{\mathbb{Z}}^3}$ is exactly the collection of loop clusters of $\mathcal{L}_{1/2}^{\mathbb{R}^3}$. Moreover, the dimension of these clusters equals $\frac{5}{2}$. 
  \end{conjecture}

As stated in \cite[Sections 3 and 4]{werner2021clusters}, this conjecture stems from the following two ingredients concerning the loop soup $\mathcal{L}_{1/2}^{\widetilde{\mathbb{Z}}^3}$: 
\begin{enumerate}

	\item[(a)]   As the mesh size tends to zero, the contribution of microscopic loops to the construction of macroscopic clusters vanishes. In other words, the scaling limit of loop clusters is governed solely by macroscopic loops.

	\item[(b)]   Based on the coupling between random walks and Brownian motion, it was shown in \cite{lawler2007random, sapozhnikov2018brownian} that the macroscopic loops in $\mathcal{L}_{1/2}^{\widetilde{\mathbb{Z}}^3}$ can be approximated bijectively by the continuum loops in $\mathcal{L}_{1/2}^{\mathbb{R}^3}$. Combined with the fact that the loops in $\mathcal{L}_{1/2}^{\mathbb{R}^3}$ can intersect (since Brownian motions on $\mathbb{R}^3$ do so; see e.g., \cite[Theorem 9.1]{morters2010brownian}), this approximation suggests that the clusters consisting of macroscopic loops converge to the loop clusters of $\mathcal{L}_{1/2}^{\mathbb{R}^3}$.


   

\end{enumerate} 
In fact, the analogue of Conjecture \ref{conj_3d} in two dimensions has been proved. Specifically, the celebrated work \cite{sheffield2012conformal} introduced the conformal loop ensembles $\mathrm{CLE}_\kappa$ for $\kappa\in (\frac{8}{3},4]$, defined as the boundaries of clusters of the Brownian loop soup in a two-dimensional domain with intensity $\frac{(3\kappa-8)(6-\kappa)}{4\kappa}$. It was also shown in \cite{sheffield2012conformal} that $\mathrm{CLE}_\kappa$ consists of loops in the form of Schramm-Loewner evolution (SLE). Building on this work, it was confirmed in \cite{lupu2019convergence} that on the discrete half-plane or its metric graph, the scaling limit of the boundaries of clusters of a loop soup with intensity $\frac{1}{2}$ is $\mathrm{CLE}_4$, which is exactly the continuum counterpart at the same intensity.


 Back to the three-dimensional case, while Ingredient (b) is natural and supported by classical couplings between random walks and Brownian motion, Ingredient (a) is more subtle and appears to be decisive for the validity of Conjecture \ref{conj_3d}. However, as outlined below, a series of results established in the companion papers \cite{inpreparation_twoarm, inpreparation_pivotal} constitute evidence against the vanishing of the contribution from microscopic loops. In \cite{inpreparation_twoarm}, the exact order of the following two-arm probability was derived. Specifically, for any $v_1,v_2\in \widetilde{\mathbb{Z}}^d$ and $N\ge 1$, we define $\mathsf{H}_{v_2}^{v_1}(N)$ as the event that there exist two distinct loop clusters $\mathcal{C}_1,\mathcal{C}_2$ of $\mathcal{L}_{1/2}^{\widetilde{\mathbb{Z}}^d}$ such that each $\mathcal{C}_i$ connects $v_i$ and $\partial B(N)$. According to \cite[Theorem 1.1]{inpreparation_twoarm}, for any $d\ge 3$ with $d\neq 6$, there exist constants $C,c>0$ depending only on $d$ such that for any $N\ge C$ and any edge $\{x,y\}\in \mathbb{L}^d$ with $x,y\in B(cN)$, 
 \begin{equation}
 	\mathbb{P}\big(\mathsf{H}_{y}^{x}(N)  \big) \asymp N^{-[(\frac{d}{2}+1)\land 4]}. 
 \end{equation}
 It follows directly that the expected number of edges in a box of size $N$, whose endpoints belong to two distinct loop clusters with diameters at least $N$, is of order $N^{(\frac{d}{2}-1)\vee (d-4)}$. As shown in \cite[Theorem 1.1]{inpreparation_pivotal}, this estimate, after a delicate second moment analysis, implies that two macroscopic loop clusters can nearly touch at many locations (in sharp contrast to the two-dimensional case, where the distance between two macroscopic loop clusters remains macroscopic; see \cite{sheffield2012conformal, lupu2019convergence}). By viewing the two clusters that nearly touch as originating from two distant sets $A_1$ and $A_2$, each touching location can naturally be interpreted as a potential interval that contains a pivotal loop for the event $\big\{A_1\xleftrightarrow{\cup \mathcal{L}_{1/2}^{\widetilde{\mathbb{Z}}^d}} A_2\big\}$. Making this interpretation precise, it was proved in \cite[Theorem 1.6]{inpreparation_pivotal} that for the one-arm event $\big\{\bm{0}\xleftrightarrow{\cup \mathcal{L}_{1/2}^{\widetilde{\mathbb{Z}}^d}} \partial B(N)\big\}$, the typical number of such intervals is of order $N^{(\frac{d}{2}-1)\land  2}$. In other words, a macroscopic loop cluster typically contains a polynomial number of cut loops (i.e., loops whose removal disconnects the cluster) at scale $1$, which is incompatible with Ingredient (a).

  The main result of this paper develops the counter-evidence established in \cite{inpreparation_twoarm, inpreparation_pivotal}, and proves that Conjecture \ref{conj_3d} does not hold. For any $x\in \mathbb{R}^d$ and $r>0$, let $\bm{\mathcal{B}}_x(r)$ denote the Euclidean ball of radius $r$ centered at $x$, and let $\bm{\mathcal{C}}_x(r)$ denote its boundary; when $x=\bm{0}$, we simply write $\bm{\mathcal{B}}(r)$ and $\bm{\mathcal{C}}(r)$. For any $A\subset \mathbb{R}^d$, we denote its upper box-counting dimension by $\overline{\mathrm{dim}}_{\mathrm{B}}(A)$. For any $r>0$, we denote by $\mathfrak{C}_r$ the collection of loop clusters of $\mathcal{L}_{1/2}^{\mathbb{R}^3}$ with diameter at least $r$. Here the diameter of a set $A\subset \mathbb{R}^d$ is defined by $\sup_{x,y\in A}\|x-y\|$.

\begin{theorem}\label{theorem_main}
	(1) There exists $\zeta\in (\frac{1}{2},1]$ such that for any $\epsilon\in (0,1)$, 
	\begin{equation}\label{ineq_main1}
		\mathbb{P}\big( \bm{\mathcal{C}}(\epsilon) \xleftrightarrow{\cup \mathcal{L}_{1/2}^{\mathbb{R}^3}} \bm{\mathcal{C}}(1) \big) = \epsilon^{\zeta+o(1)}.  
	\end{equation}

	(2)  For any $x\in \mathbb{R}^3$ and $r_1,r_2>0$, 
	\begin{equation}\label{ineq_main_2}
 	\mathbb{P}\big( \overline{\mathrm{dim}}_{\mathrm{B}}\big(\mathcal{C}\cap \bm{\mathcal{B}}_x(r_1)\big)\le 3-\zeta, \forall  \mathcal{C}\in \mathfrak{C}_{r_2}\big)  =1. 
	\end{equation}
	Consequently, the upper box-counting dimension of every loop cluster of the Brownian loop soup $\mathcal{L}_{1/2}^{\mathbb{R}^3}$ is strictly less than $\frac{5}{2}$.

\end{theorem}

\begin{remark}
	(1) Since the Hausdorff dimension of a set must be upper-bounded by its upper box-counting dimension, it follows from Theorem \ref{theorem_main} that the Hausdorff dimension of every loop cluster of the Brownian loop soup $\mathcal{L}_{1/2}^{\mathbb{R}^3}$ is at most $3-\zeta$, and hence is strictly less than $\frac{5}{2}$. We expect that $3-\zeta$ is also a lower bound for both the Hausdorff dimension and the lower box-counting dimension. Establishing these bounds appears delicate and is beyond the scope of this paper.



	\noindent (2) When $d\ge 4$, Brownian motions in $\mathbb{R}^d$ do not intersect (see \cite[Theorem 9.1]{morters2010brownian}). As a result, every cluster of the Brownian loop soup $\mathcal{L}_{1/2}^{\mathbb{R}^d}$ is a single loop, and thus almost surely has Hausdorff and box-counting dimensions both equal to $2$ (since the Brownian motion does so; see e.g., \cite[Theorem 16.2]{falconer2013fractal}).


\end{remark}

We now turn to the discrete model $\mathcal{L}_{1/2}^{\widetilde{\mathbb{Z}}^3}$ featured in Conjecture \ref{conj_3d}. As shown independently in \cite[Theorem 1.2]{cai2024incipient} and \cite[Corollary 1.2]{drewitz2024cluster}, the volume of a loop cluster of $\mathcal{L}_{1/2}^{\widetilde{\mathbb{Z}}^d}$ at scale $N$ is typically of order $N^{(\frac{d}{2}+1)\land  4}$. Based on this estimate, it is plausible that the scaling limit of loop clusters of $\mathcal{L}_{1/2}^{\widetilde{\mathbb{Z}}^d}$ has fractal dimension $(\frac{d}{2}+1)\land  4$, provided this limit exists. Next, we present a rigorous statement of this property below the critical dimension $d_c=6$. (\textbf{P.S.} For $d<6$, the typical number of macroscopic clusters in a box of size $N$ is of order $1$, whereas for $d>6$ it is of order $N^{d-6} \gg 1$. This fundamental difference reveals that the scaling limits below and above the critical dimension belong to distinct universality classes.) For any $x\in \mathbb{R}^d$ and $r>0$, we denote by $\Lambda_x(r):=x+[-r,r]^d$ the box of side length $2r$ centered at $x$; when $x=\bm{0}$, we abbreviate $\Lambda_{\bm{0}}(r)$ as $\Lambda(r)$. For any $\delta >0$, let $\delta \cdot \widetilde{\mathbb{Z}}^d$ denote the graph obtained by scaling the metric graph $\widetilde{\mathbb{Z}}^d$ by a factor of $\delta$ and then embedding it into $\mathbb{R}^d$. We define $\delta \cdot \mathcal{L}_{\alpha}^{\widetilde{\mathbb{Z}}^d}$ as the loop soup on $\delta \cdot \widetilde{\mathbb{Z}}^d$ of intensity $\alpha$. In this paper, we say that the collection of loop clusters of $\mathcal{L}_{1/2}^{\widetilde{\mathbb{Z}}^d}$ has a scaling limit $\mathcal{S}^*$ if the following hold: 
\begin{enumerate}

	\item[(I)]  $\mathcal{S}^*$ is a collection of countably many subsets of $\mathbb{R}^d$;

	\item[(II)] Arbitrarily fix $k\in \mathbb{N}^+$, $\{x_i\}_{1\le i\le k}\subset \mathbb{R}^d$, $\{r_i\}_{1\le i\le k}\subset \mathbb{R}^+$ and $r'>0$. For $1\le i\le k$, let $\mathbf{M}_i^{(n)}$ (resp. $\mathbf{M}_i^*$) be the number of loop clusters of $n^{-1}\cdot \mathcal{L}_{1/2}^{\widetilde{\mathbb{Z}}^d}$ (resp. elements of $\mathcal{S}^*$) that have diameter at least $r'$ and intersect $\Lambda_{x_i}(r_i)$. Then $\big(\mathbf{M}_1^{(n)},...,\mathbf{M}_k^{(n)}\big)$ converges in distribution to $\big(\mathbf{M}_1^{*},...,\mathbf{M}_k^{*}\big)$ as $n\to \infty$.

\end{enumerate}

  \begin{theorem}\label{theorem_main_part2}
  	For any $3\le d\le 5$, if the collection of loop clusters of $\mathcal{L}_{1/2}^{\widetilde{\mathbb{Z}}^d}$ has a scaling limit $\mathcal{S}^*$, then $\mathcal{S}^*$ satisfies the following two properties: 
  \begin{enumerate}
  	\item   For any $\epsilon\in (0,1)$, 
  	\begin{equation}\label{ineq_main_part2_1}
  			\mathbb{P}\big( \bm{\mathcal{C}}(\epsilon) \xleftrightarrow{  \mathcal{S}^*} \bm{\mathcal{C}}(1) \big) \asymp  \epsilon^{\frac{d}{2}-1}.  
  	\end{equation}

  	\item	 For any $x\in \mathbb{R}^d$ and $r_1,r_2>0$,  	\begin{equation}\label{ineq_main_part2}
  \mathbb{P}\big( \exists \mathcal{C}^* \in \mathfrak{C}_{r_2}^*\ \text{such that}\ 
 \overline{\mathrm{dim}}_{\mathrm{B}}( \mathcal{C}^* \cap \bm{\mathcal{B}}_x(r_1)) =\tfrac{d}{2}+1\big) >0. 
	\end{equation} 
	Here $\mathfrak{C}_{r_2}^*$ denotes the collection of elements of $\mathcal{S}^*$ with diameter at least $r_2$.

  \end{enumerate}

  \end{theorem}

  \begin{remark}
(1) Since $\zeta>\frac{d}{2}-1$ and $3-\zeta<\frac{d}{2}+1$ for $\zeta\in (\frac{1}{2},1]$ and $d=3$, the one-arm exponent in (\ref{ineq_main1}) does not match that in (\ref{ineq_main_part2_1}), and the property of $\mathfrak{C}_{r_2}$ in (\ref{ineq_main_2}) contradicts that of $\mathfrak{C}_{r_2}^*$ in (\ref{ineq_main_part2}). As a result, Theorems \ref{theorem_main} and \ref{theorem_main_part2} together imply that if the scaling limit of loop clusters of $\mathcal{L}_{1/2}^{\widetilde{\mathbb{Z}}^3}$ exists (in the sense described before Theorem \ref{theorem_main_part2}), this limit cannot be the loop clusters of $\mathcal{L}_{1/2}^{\mathbb{R}^3}$.


\noindent (2)  In fact, the definition of scaling limit given before Theorem \ref{theorem_main_part2} is rather weak. Usually, convergence of a discrete percolation model to a continuum one requires that after rescaling, for every $r>0$, with high probability there exists a bijection between their clusters with diameter at least $r$ such that the Hausdorff distance between each corresponding pair of clusters vanishes (see e.g., \cite{camia2006two} and \cite{lupu2019convergence} for the critical Bernoulli percolation and the critical loop soup respectively, both on two-dimensional graphs). Such a requirement immediately implies the conditions for the scaling limit in this paper.

	\noindent (3) It is clear from our proof in Section \ref{section5_S_star} that Theorem \ref{theorem_main_part2} also holds whenever a subsequential scaling limit exists. For simplicity, we state this theorem assuming the existence of the scaling limit.



\end{remark}

%

Let $\mathfrak{L}$ denote the point measure consisting of the loops in $\mathcal{L}_{1/2}^{\widetilde{\mathbb{Z}}^d}$ that are contained in a single interval $I_e$ for some $e\in \mathbb{L}^d$ and intersect both of its trisection points. The key to proving Theorem \ref{theorem_main} is to show that removing all loops in $\mathfrak{L}$ (which are at scale $1$) will significantly reduce the crossing probability; see (\ref{crossing_low}) and (\ref{crossing_high}) for the bounds without removal. Although the derivation of Theorem \ref{theorem_main} requires only the three-dimensional case, we still present this result in full generality, as the higher-dimensional cases are interesting in their own right.

\begin{theorem}\label{theorem_const_removal_edge_loop}
For any $d\ge 3$ with $d\neq 6$, there exists a constant $c>0$ depending only on $d$ such that for any $N > n \ge 1$,
     \begin{align}
	&\text{when}\ 3\le d<6,\ \  \mathbb{P}\big( B(n) \xleftrightarrow{\cup (  \mathcal{L}_{1/2}^{\widetilde{\mathbb{Z}}^d}- \mathfrak{L} ) }  \partial B(N) \big)\lesssim       \big(n/N \big)^{\frac{d}{2}-1+c}; \label{goodineq_128} \\ 
	&\text{when}\ d>6, \ \ \ \ \   \ \  \mathbb{P}\big( B(n) \xleftrightarrow{\cup (  \mathcal{L}_{1/2}^{\widetilde{\mathbb{Z}}^d}- \mathfrak{L} ) }  \partial B(N) \big)\lesssim       \big( n^{\frac{d-4}{2}} /N  \big)^{2+c}. \label{goodineq_129}
	 \end{align}
\end{theorem}

   \begin{remark}\label{remark1.7_Zd}
	Under the coupling between $\mathcal{L}_{1/2}^{\mathbb{Z}^d}$ and $\mathcal{L}_{1/2}^{\widetilde{\mathbb{Z}}^d}$, every fundamental loop in $\mathcal{L}_{1/2}^{\mathbb{Z}^d}$ is contained in its corresponding loop in $\mathcal{L}_{1/2}^{\widetilde{\mathbb{Z}}^d}$. Moreover, all these loops in $\mathcal{L}_{1/2}^{\widetilde{\mathbb{Z}}^d}$ belong to $\mathcal{L}_{1/2}^{\widetilde{\mathbb{Z}}^d}-\mathfrak{L}$, since $\mathfrak{L}$ consists only of edge loops. As a result, $\cup \mathcal{L}_{1/2}^{\mathbb{Z}^d}$ is stochastically dominated by $\cup (  \mathcal{L}_{1/2}^{\widetilde{\mathbb{Z}}^d}- \mathfrak{L} )$, which implies that the analogues of (\ref{goodineq_128}) and (\ref{goodineq_129}) with $\mathcal{L}_{1/2}^{\widetilde{\mathbb{Z}}^d}- \mathfrak{L}$ replaced by $\mathcal{L}_{1/2}^{\mathbb{Z}^d}$ also hold. This improves the previous bounds in \cite[Theorem 1.3]{chang2016phase}.

\end{remark}


To conclude this section, we discuss two possible sources of the gap between the cluster dimensions of $\mathcal{L}_{1/2}^{\mathbb{R}^3}$ and $\mathcal{L}_{1/2}^{\widetilde{\mathbb{Z}}^3}$, as well as some potential research directions. The first possibility has been pointed out in \cite{werner2021clusters}. Specifically, recall that when $d\ge 4$, all loops in $\mathcal{L}_{1/2}^{\mathbb{R}^d}$ are disjoint from each other, and hence every loop cluster of $\mathcal{L}_{1/2}^{\mathbb{R}^d}$ has dimension $2$. Meanwhile, the scaling limit of loop clusters of $\mathcal{L}_{1/2}^{\widetilde{\mathbb{Z}}^d}$ for $d\in \{4,5\}$ is expected to be $(\frac{d}{2}+1)$-dimensional. To bridge this dimensional gap, it was conjectured in \cite{werner2021clusters} that when $d\in \{4,5\}$, this scaling limit can be recovered by imposing a gluing mechanism between loops in $\mathcal{L}_{1/2}^{\mathbb{R}^d}$; moreover, this mechanism was conjectured to be  deterministic (i.e., measurable with respect to $\mathcal{L}_{1/2}^{\mathbb{R}^d}$) when $d=4$, and to be stochastic (i.e., requiring extra randomness beyond $\mathcal{L}_{1/2}^{\mathbb{R}^d}$) when $d=5$. By Theorems \ref{theorem_main} and \ref{theorem_main_part2}, this gluing mechanism may also be required in the three-dimensional case (though it was initially conjectured otherwise; see Conjecture \ref{conj_3d}). In summary, the gap between cluster dimensions might be characterized by a gluing mechanism, which essentially captures the contribution of microscopic loops to the construction of macroscopic loop clusters.

     Another possibility is that the aforementioned gap cannot be filled by gluing loops, but instead requires increasing the intensity of the loop soup. This naturally leads to the following question: 
       \begin{center}
  	\textit{Is the critical intensity of $\mathcal{L}_{\alpha}^{\mathbb{R}^3}$ equal to $\frac{1}{2}$? }
  \end{center}
  By analogy, it is also compelling to ask whether $\mathcal{L}_{\alpha}^{\mathbb{Z}^3}$ has critical intensity $\frac{1}{2}$. Next, we briefly review the relevant existing results. Recall that it has been verified in \cite{lupu2016loop, chang2024percolation} that for any $d\ge 3$, the critical intensity of $\mathcal{L}_{\alpha}^{\widetilde{\mathbb{Z}}^d}$ equals $\frac{1}{2}$. Note that the proof of this equality mainly relies on the formula (\ref{two_point_function}) for the two-point function, which has no direct analogue in $\mathbb{R}^3$ or $\mathbb{Z}^3$. For $\mathcal{L}_{\alpha}^{\mathbb{R}^3}$, it follows from Theorem \ref{theorem_main} that $\mathcal{L}_{1/2}^{\mathbb{R}^3}$ does not percolate, and thus its critical intensity is at least $\frac{1}{2}$. The same lower bound holds for $\mathcal{L}_{\alpha}^{\mathbb{Z}^3}$, since its range is stochastically dominated by that of $\mathcal{L}_{\alpha}^{\widetilde{\mathbb{Z}}^3}$. To sum up, it remains to confirm whether there exists $\epsilon>0$ such that $\mathcal{L}_{1/2+\epsilon}^{\mathbb{R}^3}$ (or $\mathcal{L}_{1/2+\epsilon}^{\mathbb{Z}^3}$) does not percolate. Notably, in the two-dimensional case, such an $\epsilon>0$ is absent: it has been proved in \cite{lupu2016loop2d} that for any $\alpha>\frac{1}{2}$, the loop soup of intensity $\alpha$ on the discrete half-plane $\mathbb{Z}\times \mathbb{N}$ percolates.

  To the best of our knowledge, no mathematical evidence exists to support either of these two possibilities. This requires further research to determine.

    \textbf{Organization of this paper.} Section \ref{section_pre} reviews preliminary results on the loop measure $\mu^{\Omega}$ and the loop soup $\mathcal{L}_{\alpha}^{\Omega}$ for $\Omega\in \{\mathbb{Z}^d, \mathbb{R}^d, \widetilde{\mathbb{Z}}^d\}$. Section \ref{section_proof_crossing_probability} presents the proof of Theorem \ref{theorem_const_removal_edge_loop}. Our main result, Theorem \ref{theorem_main}, is then established in Section \ref{section_existence_one_arm}. Finally, Section \ref{section5_S_star} is devoted to the proof of Theorem \ref{theorem_main_part2}.

 \textbf{Statements about constants.} We adhere to the following conventions:

 \begin{itemize}

 	\item The letters $C$ and $c$ denote positive constants that may vary according to the context. Here the letter $C$ (with or without subscripts) represents large constants, while $c$ represents small ones.


 	\item Unless stated otherwise, a constant depends only on the dimension $d$.


 \end{itemize}

\section{Preliminaries}\label{section_pre}

In this section, we review some useful properties of the loop measure $\mu^{\Omega}$ and the loop soup $\mathcal{L}_{\alpha}^{\Omega}$ for $\Omega\in \{\mathbb{Z}^d, \mathbb{R}^d, \widetilde{\mathbb{Z}}^d\}$.

\textbf{Embedding paths into $\mathbb{R}^d$.} Unless otherwise specified, we consider $\mathbb{Z}^d$ as a subset of $\mathbb{R}^d$ under the canonical embedding. This allows one to regard a discrete path $\eta=(x_0, \ldots, x_k)$ (where $k\in \mathbb{N}^+$) as a continuous curve in $\mathbb{R}^d$ through linear interpolation, i.e., $\eta(t):=  (1-t )\cdot x_i+t \cdot x_{i+1}$ for $0\le i\le k-1$ and $i \le t\le  i+1$. The canonical embedding of $\mathbb{Z}^d$ also induces an embedding of $\widetilde{\mathbb{Z}}^d$ into $\mathbb{R}^d$, where each interval $I_e$ is scaled by $\frac{1}{d}$. This embedding identifies any path on $\widetilde{\mathbb{Z}}^d$ with a curve in $\mathbb{R}^d$. In particular, any loop on $\mathbb{Z}^d$ or $\widetilde{\mathbb{Z}}^d$ can be viewed as a curve in $\mathbb{R}^d$.

 \textbf{Coupling between discrete and continuum loop soups.} We now review the coupling in \cite{sapozhnikov2018brownian} between $\mathcal{L}_{\alpha}^{\mathbb{Z}^d}$ and $\mathcal{L}_{\alpha}^{\mathbb{R}^d}$. We first record some necessary notations. For any curve $\eta$ in $\mathbb{R}^d$, we denote its duration by $T_{\eta}$. For two curves $\eta_1$ and $\eta_2$ in $\mathbb{R}^d$, their Fr\'echet distance is defined by 
\begin{equation}
	\mathrm{d}_{\mathrm{F}}(\eta_1,\eta_2):= \inf_{\varphi_1,\varphi_2} \sup_{t\in [0,1]} \big\| \eta_1\big(\varphi_1(t)\big)-\eta_2\big(\varphi_2(t)\big)  \big\|,
\end{equation}
where each $\varphi_i$ is a continuous, non-decreasing surjection from $[0,1]$ to $[0,T_{\eta_i}]$. For any $N \ge 1$, we define the following operator $\varphi_{N}$ acting on curves in $\mathbb{R}^d$: 
 \begin{equation}
 	\varphi_N[\eta](t)= N^{-1}\cdot  \eta(dN^2t),  \ \ \forall t\in \big[0, (dN^2)^{-1}T_{\eta}\big].  
 	 \end{equation}
  Here the scaling factors are chosen in accordance with the invariance principle: the covariance of a simple random walk on $\mathbb{Z}^d$ after $dN^2$ steps is $N^2\cdot \mathbb{I}_d$, where $\mathbb{I}_d$ is the $d$-dimensional identity matrix. For any $\alpha,r,\theta>0$ and $N\ge 1$, let
  \begin{equation}\label{mathcalA_Zd}
	\mathcal{A}_{\alpha,r,\theta,N}^{\mathbb{Z}^d} := \big\{ \ell'= \varphi_N[\ell]: \ell\in  \mathcal{L}_{\alpha}^{\mathbb{Z}^d}, T_{\ell}>  N^{\theta} , N^{-1}\cdot  \ell(0) \in \Lambda( r )  \big\}.
\end{equation} 
 Similarly, for the Brownian loop soup $\mathcal{L}_{\alpha}^{\mathbb{R}^d}$, we define  \begin{equation}\label{mathcalA_Rd}
	\mathcal{A}_{\alpha,r,\theta,N}^{\mathbb{R}^d} := \big\{   \ell\in  \mathcal{L}_{\alpha}^{\mathbb{R}^d}:	 T_{ \ell}  > \big(\tfrac{N^\theta}{d}-\tfrac{d+4}{2d(d+2)}\big) N^{-2} ,  	 \ell(0) \in \Lambda( r+ \tfrac{1}{2N} ) \big\}. 
\end{equation}



  \begin{lemma}[{\cite[Theorem 2.2]{sapozhnikov2018brownian}}]\label{lemma_coupling}
 	For any $d\ge 2$ and $\alpha>0$, there exists a constant $C>0$ and a coupling of $\mathcal{L}_{\alpha}^{\mathbb{Z}^d}$ and $\mathcal{L}_{\alpha}^{\mathbb{R}^d}$ such that the following holds. For any $r\ge 1$, $N\ge 1$ and $\theta\in (\frac{2d}{d+4},2)$, with probability at least $1-C(\alpha+1)r^dN^{-\min\{\frac{d}{2}, \theta(\frac{d}{2}+2)-d \}}$, there exists a bijection $\varpi_{\alpha,r,\theta,N}:\mathcal{A}_{\alpha,r,\theta,N}^{\mathbb{R}^d}\to \mathcal{A}_{\alpha,r,\theta,N}^{\mathbb{Z}^d}$ such that  
 	 \begin{equation}\label{final2.8}
 	\mathrm{d}_{\mathrm{F}}\big(\ell,\varpi_{\alpha,r,\theta,N}(\ell)\big) \le CN^{-\frac{1}{4}}\mathrm{ln}(N), \   \forall \ell\in \mathcal{A}_{\alpha,r,\theta,N}^{\mathbb{R}^d}. 
 \end{equation} 
 	 	  \end{lemma}

\textbf{Crossing loops.} We recall the following useful estimate on the total mass of Brownian loops that cross an annulus.

 \begin{lemma}\label{lemma_crossing_loop_mass}
	For any $d\ge 3$, $r>0$ and $\lambda>1$, 
	   \begin{equation}\label{mass_crossing_loop}
 		 \mu^{\mathbb{R}^d}\big( \big\{ \ell: \ell \ \text{intersects}\ \bm{\mathcal{C}}(r)\ \text{and}\ \bm{\mathcal{C}}(\lambda  r) \big\} \big)\asymp \lambda^{2-d}. 
 	\end{equation}
 	\end{lemma}

 We refer to \cite[Lemma 2.7]{chang2016phase} for the analogue of this lemma on $\mathbb{Z}^d$, and to \cite[Lemma 2.10]{cai2024quasi} for that on $\widetilde{\mathbb{Z}}^d$. To extend either result to $\mathbb{R}^d$, it suffices to replace the role of the hitting probabilities of random walks with the hitting densities of Brownian motion. We therefore omit this straightforward extension.

 Lemma \ref{lemma_crossing_loop_mass} provides the following lower bound for the one-arm probability:  
 \begin{equation}\label{newadd210}
	\begin{split}
		  \mathbb{P}\big( \bm{\mathcal{C}}(\epsilon) \xleftrightarrow{\cup \mathcal{L}_{1/2}^{\mathbb{R}^3}} \bm{\mathcal{C}}(1) \big)  
		\ge  & 	\mathbb{P}\big( \exists \ell \in \mathcal{L}_{1/2}^{\mathbb{R}^3}  \ \text{intersecting}\ \bm{\mathcal{C}}(\epsilon)\ \text{and}\ \bm{\mathcal{C}}(1)   \big) \asymp \epsilon.
	\end{split}
\end{equation}
This bound implies that the one-arm exponent $\zeta\le 1$.

\begin{remark}[strict upper bound for $\zeta$]\label{remark_zeta_<1}
As we sketch below, the bound $\zeta\le 1$ can be strengthened to $\zeta<1$, i.e., there exists a constant $c>0$ such that  
\begin{equation}\label{improve_2.8}
  \mathbb{P}\big( \bm{\mathcal{C}}(\epsilon)
  \xleftrightarrow{\cup \mathcal{L}_{1/2}^{\mathbb{R}^3}}
  \bm{\mathcal{C}}(1) \big)
  \gtrsim \epsilon^{1-c}, \ \forall \epsilon\in (0,1). 
\end{equation}
We now justify (\ref{improve_2.8}). Take a sufficiently large constant $\lambda>1$. For any $i\in \mathbb{N}^+$ and $j\in \{-1,0,1\}$, we denote $r_i^{(j)}:=\lambda^{3i+j}\epsilon$. Let $i_\star:=\min\{i: r_{i+2}^{(0)}\ge 1  \}$, and note that $i_\star\asymp \mathrm{ln}(\epsilon^{-1})$. We say that the loops of $\mathcal{L}_{1/2}^{\mathbb{R}^3}$ in the annulus $\bm{\mathcal{B}}(r_{i+1}^{(0)})\setminus \bm{\mathcal{B}}(r_i^{(0)})$ are good (denoted by $\mathsf{G}_i$) if the following holds: conditioned on these loops, the probability for two independent Brownian motions $\eta_{i}$ and $\eta_{i}'$ to be connected via them is at least $\frac{1}{2}$, where $\eta_{i}$ (resp. $\eta_{i}'$) starts from $\bm{\mathcal{C}}(r_{i}^{(1)})$ (resp. $\bm{\mathcal{C}}(r_{i+1}^{(-1)})$) and is stopped upon hitting $\bm{\mathcal{C}}(r_{i}^{(0)})$ (resp. $\bm{\mathcal{C}}(r_{i+1}^{(0)})$). By construction, $\mathsf{G}_i$ depends only on the loops contained in $\bm{\mathcal{B}}(r_{i+1}^{(0)})\setminus \bm{\mathcal{B}}(r_i^{(0)})$. Since these annuli are disjoint, the events $\{\mathsf{G}_i\}_{i\in \mathbb{N}^+}$ are independent. In addition, since Brownian loops in $\mathbb{R}^3$ have positive capacities, the probability of $\mathsf{G}_i$ is uniformly bounded away from zero. Therefore, it follows from Hoeffding's inequality that for a sufficiently small $c'>0$,  
\begin{equation}\label{improve_2.9}
	\mathbb{P}\big(\overline{\mathsf{G}}\big):=  \mathbb{P}\Big(\sum\nolimits_{1\le i\le i_\star} \mathbbm{1}_{\mathsf{G}_i}\ge c'\mathrm{ln}(\epsilon^{-1})\Big) \ge  \tfrac{1}{2}.		\end{equation}
 As shown in Lemma \ref{lemma_crossing_loop_mass}, the energy cost of crossing the annulus $\bm{\mathcal{B}}(R)\setminus \bm{\mathcal{B}}(r)$ is of order $r/R$. Thus, if the loops in $\bm{\mathcal{B}}(r_{i+1}^{(0)})\setminus \bm{\mathcal{B}}(r_i^{(0)})$ are good, then one may replace a loop crossing this annulus by two loops crossing
$\bm{\mathcal{B}}(r_{i}^{(1)})\setminus \bm{\mathcal{B}}(r_i^{(0)})$ and $\bm{\mathcal{B}}(r_{i+1}^{(0)})\setminus \bm{\mathcal{B}}(r_{i+1}^{(-1)})$ respectively, and connect them via these good loops; this modification reduces the total energy cost by a factor of order $\lambda^{-1}$. Consequently, given that $\overline{\mathsf{G}}$ occurs, the conditional probability of $\big\{\bm{\mathcal{C}}(\epsilon)
  \xleftrightarrow{\cup \mathcal{L}_{1/2}^{\mathbb{R}^3}}
  \bm{\mathcal{C}}(1)\big\}$ is at least 
  \begin{equation}
  \begin{split}
  	  \mathbb{P}\big( \exists \ell \in \mathcal{L}_{1/2}^{\mathbb{R}^3}  \ \text{intersecting}\ \bm{\mathcal{C}}(\epsilon)\ \text{and}\ \bm{\mathcal{C}}(1)   \big) \cdot (c''\lambda)^{c'\mathrm{ln}(\epsilon^{-1})} \gtrsim \epsilon^{1-c}. 
  \end{split}
  \end{equation}
Combined with (\ref{improve_2.9}), it implies the strengthened bound (\ref{improve_2.8}). 
\end{remark}

 \textbf{Decomposition of loops.} In this part, we review a decomposition argument for crossing loops. This argument was first applied to loops on the metric graph $\widetilde{\mathbb{Z}}^d$ (see \cite[Section 2.6.3]{cai2024high}), and we adapt it here to Brownian loops in $\mathbb{R}^d$. For any $R>r>0$, we denote by $\mathfrak{L}[r,R]$ the point measure consisting of loops in $\mathcal{L}_{1/2}^{\mathbb{R}^d}$ that intersect both $\bm{\mathcal{C}}(r)$ and $\bm{\mathcal{C}}(R)$. We decompose the loops in $\mathfrak{L}[r,R]$ via the following procedure. For any $\ell \in \mathfrak{L}[r,R]$, we choose an arbitrary time-shift $\ell'$ of $\ell$ satisfying the following properties: 
 \begin{itemize}

 	\item  $\ell'(0)\in \bm{\mathcal{C}}(r)$;

 	\item  There exists $t \in (0, T_{\ell'})$ such that $\ell'(t) \in \bm{\mathcal{C}}(R)$ and $\ell'(s)\in \bm{\mathcal{B}}(R)\setminus \bm{\mathcal{B}}(r)$ for all $s\in (t,T_{\ell'})$.

 \end{itemize}
 We define a sequence $\{\widetilde{\tau}_j\}_{0\le j\le 2\xi}$ of stopping times recursively as follows:
 \begin{enumerate}

 	\item[(i)]  We set $\widetilde{\tau}_0=0$;

 	\item[(ii)]  For each $j\in \mathbb{N}$, given $\widetilde{\tau}_{2j}$, let $\widetilde{\tau}_{2j+1}:=\inf\{ t\in (\widetilde{\tau}_{2j}, T_{\ell'}): \ell'(t)\in \bm{\mathcal{C}}(R) \}$;

    \item[(iii)] Given $\widetilde{\tau}_{2j+1}$, let $\widetilde{\tau}_{2j+2}:=\inf\{ t\in (\widetilde{\tau}_{2j+1}, T_{\ell'}]: \ell'(t)\in \bm{\mathcal{C}}(r)\}$. If $\widetilde{\tau}_{2j+2}=T_{\ell'}$, we stop the construction and set $\xi:=j+1$; otherwise, we proceed to define $\widetilde{\tau}_{2j+3}$ as in (ii).




 	
 \end{enumerate}
 For each $1\le j\le \xi$, we define the $j$-th forward (resp. backward) crossing path, denoted by $\eta^{1}_j$ (resp. $\eta^{-1}_j$), as the restriction of $\ell'$ to the time interval $[\widetilde{\tau}_{2j-2}, \widetilde{\tau}_{2j-1}]$ (resp. $[\widetilde{\tau}_{2j-1}, \widetilde{\tau}_{2j}]$). Note that $\xi$ and $\{\eta^{i}_j\}_{1\le j\le \xi,i\in \{1,-1\}}$ are invariant with respect to the selection of $\ell'$. The construction of $\{\widetilde{\tau}_j\}_{0\le j\le 2\xi}$ ensures that the ending point of each forward crossing path is the starting point of a backward crossing path, and vice versa. We define $\kappa:= \sum_{\ell \in \mathfrak{L}[r,R] } \xi(\ell)$ as the total number of crossings in the loops of $\mathfrak{L}[r,R]$. For simplicity, we enumerate all crossing paths of the loops in $\mathfrak{L}[r,R]$ as $\{\eta^{i}_j\}_{1\le j\le \kappa,i\in \{1,-1\}}$. In addition, we define $\mathcal{F}_{\eta}^{\pm}$ as the $\sigma$-field generated by the starting and ending points of $\{\eta^{1}_j\}_{1\le j\le \kappa}$. This decomposition admits the following two useful properties:  
   \begin{itemize}

     	\item[$\diamond$]  It follows from Lemma \ref{lemma_crossing_loop_mass} that the number of loops in $\mathfrak{L}[r,R]$ decays exponentially. Moreover, for any loop $\ell \in \mathfrak{L}[r,R]$, its crossing number $\xi(\ell)$ also decays exponentially. This is because the energy cost of each backward crossing path (i.e., the probability for the diffusion $X_\cdot^{\mathbb{R}^d}$ starting from $\bm{\mathcal{C}}(R)$ to hit $\bm{\mathcal{C}}(r)$) is bounded away from $1$. Taken together, these two facts yield the exponential decay of $\kappa$ (detailed calculations can be found in \cite[Lemma 2.4]{cai2024high}). To be precise, for any $d\ge 3$, there exists $C>0$ such that for any $R>r>0$ and $l\ge 1$,  
     	\begin{equation}\label{ineq_exp_decay_crossing}
	\mathbb{P}( \kappa\ge l )\le \big( \tfrac{Cr}{R} \big)^{(d-2)l}.
	\end{equation}

     	\item[$\diamond$] Conditioned on $\mathcal{F}_{\eta}^{\pm}$, the crossing paths $\{\eta^{i}_j\}_{1\le j\le \kappa,i\in \{1,-1\}}$ are independent. Moreover, for any $1\le j\le \kappa$, $y\in \bm{\mathcal{C}}(r)$ and $z\in \bm{\mathcal{C}}(R)$, when $\eta^{1}_j$ starts from $y$ and ends at $z$, its distribution is given by 
 \begin{equation}
  \mathbb{P}^{\mathbb{R}^d}_{y}\big( \{X_t^{\mathbb{R}^d}\}_{0\le t\le \tau_{\bm{\mathcal{C}}(R)} } \in \cdot  \mid  \tau_{\bm{\mathcal{C}}(R)}=\tau_{z} <\infty  \big),  
\end{equation}
where $\tau_A$ denotes the time when the diffusion $X_\cdot^{\mathbb{R}^d}$ first hits $A$. Similarly, when $\eta^{-1}_j$ starts from $z$ and ends at $y$, it is distributed as      	     	
     	\begin{equation}
  \mathbb{P}^{\mathbb{R}^d}_{z}\big( \{X_t^{\mathbb{R}^d}\}_{0\le t\le \tau_{\bm{\mathcal{C}}(r)} } \in \cdot  \mid  \tau_{\bm{\mathcal{C}}(r)}=\tau_{y} <\infty  \big).  
\end{equation}
This property is usually referred to as the ``spatial Markov property''; see \cite{werner2016spatial} for a thorough discussion.



     \end{itemize}

  \textbf{Stochastic domination for crossing paths.} As a key observation from \cite[Section 4]{cai2024quasi}, the ranges of crossing paths can be stochastically dominated by those of the loops crossing a closer annulus. To be precise, for any $r>0$ and $\delta\in (0,1)$, we denote $\mathfrak{B}^{1}_{r,\delta}:=\bm{\mathcal{B}}(\delta r)$ and $\mathfrak{B}^{-1}_{r,\delta}:=[\bm{\mathcal{B}}(\delta^{-1} r)]^c$. For any curve $\eta$ in $\mathbb{R}^d$, we denote its range (i.e., the collection of points visited by $\eta$) by $\mathrm{ran}(\eta)$.

   \begin{lemma}\label{lemma_domination}
 	 For any $d\ge 3$, there exist $c>c'>0$ such that the following holds for all $i\in \{1,-1\}$, $r_{-1}>r_{1}>0$, $x_{1}\in \bm{\mathcal{C}}(r_1)$ and $x_{-1}\in \bm{\mathcal{C}}(r_{-1})$. Suppose that the curve $\eta$ is distributed as 
 	 \begin{equation}
   \mathbb{P}^{\mathbb{R}^d}_{x_i}\big( \{X_t^{\mathbb{R}^d}\}_{0\le t\le \tau_{\bm{\mathcal{C}}(r_{-i})}} \in \cdot  \mid  \tau_{\bm{\mathcal{C}}(r_{-i})}=\tau_{x_{-i}} <\infty  \big). 
 	 \end{equation}
 	 Then $\mathrm{ran}(\eta)\cap \mathfrak{B}^{i}_{r_i,c'}$ is stochastically dominated by $(\cup\mathfrak{L}[c^ir_i,10c^ir_i])\cap \mathfrak{B}^{i}_{r_i,c'}$. 
 \end{lemma}

 The analogue of this domination on $\widetilde{\mathbb{Z}}^d$ was established in \cite[Lemma 4.2]{cai2024quasi}. As in Lemma \ref{lemma_crossing_loop_mass}, the argument in \cite{cai2024quasi} can be directly adapted to $\mathbb{R}^d$ by replacing the hitting probabilities of random walks with the hitting densities of Brownian motion. We therefore omit the details.

 In what follows, we present two corollaries of Lemma \ref{lemma_domination}.

   \begin{corollary}\label{coro_with_crossing_loop}
	For any $d\ge 3$, there exist $C>0$ and $c>c'>0$ such that for any $r_1>0$, $r_{-1}\ge Cr_1$, $A_1\subset \mathfrak{B}^{1}_{r_1,c'}$ and $A_{-1}\subset \mathfrak{B}^{-1}_{r_{-1},c'}$, 
	\begin{equation}\label{new215}
		\mathbb{P}\big(   A_1 \xleftrightarrow{\cup  \mathcal{L}_{1/2}^{\mathbb{R}^d}   }  A_{-1},  \mathfrak{L}[r_1,r_{-1}]\neq 0 \big)  \lesssim \big( \frac{r_1}{r_{-1}}\big)^{d-2}	\prod_{i\in \{1,-1\}}	\mathbb{P}\big(    A_i \xleftrightarrow{\cup  \mathcal{L}_{1/2}^{\mathbb{R}^d}}  \bm{\mathcal{C}}(c^i r_i) \big).  
	\end{equation}
\end{corollary}

\begin{proof}
We adapt the proof of \cite[Lemma 5.1]{cai2024quasi}, where the analogue of this lemma for $\widetilde{\mathbb{Z}}^d$ was established.

  For brevity, we denote the event on the left-hand side of (\ref{new215}) by $\mathsf{F}$. We employ the notations of crossing paths with $r=10r_1$ and $R=\frac{r_{-1}}{10}$. For $i\in \{1,-1\}$, let $\mathfrak{U}_i$ denote the union of the ranges of loops in $\mathcal{L}_{1/2}^{\mathbb{R}^d}- \mathfrak{L}[10r_1,\frac{r_{-1}}{10}]$ and crossing paths $\{\eta_j^{i}\}_{1\le j\le \kappa}$. Conditioned on $\mathcal{F}_{\eta}^{\pm}$, when $\mathsf{F}$ occurs, since $A_1$ and $A_{-1}$ lie on opposite sides of the annulus $\mathfrak{B}^{-1}_{r_{-1},c'}\setminus \mathfrak{B}^{1}_{r_1,c'}$, the event $\mathsf{G}_i:=  \big\{ A_i \xleftrightarrow{ \mathfrak{U}_i }  \bm{\mathcal{C}}(c^{i}r_i) \big\} $ must occur for all $i\in \{1,-1\}$. Note that $\mathsf{G}_i$ is measurable with respect to $\mathcal{L}_{1/2}^{\mathbb{R}^d}\cdot \mathbbm{1}_{\mathrm{ran}(\ell)\cap \mathfrak{B}^{i}_{r_{i},c}\neq \emptyset }$ and $\{\eta_j^{i}\}_{1\le j\le \kappa}$. In addition, any loop intersecting both $\mathfrak{B}^{1}_{r_{1},c}$ and $\mathfrak{B}^{-1}_{r_{-1},c}$ must be contained in $\mathfrak{L}[10r_1,\frac{r_{-1}}{10}]$, and hence has been decomposed into crossing paths. Thus, given $\mathcal{F}_{\eta}^{\pm}$, $\mathsf{G}_1$ and $\mathsf{G}_{-1}$ are conditionally independent. As a result,  
	 	  \begin{equation}\label{new216}
	  	\begin{split}
	  		\mathbb{P}\big(\mathsf{F} \big)\le & \mathbb{E}\Big[ \mathbbm{1}_{\kappa\ge 1} \cdot \prod\nolimits_{i\in \{1,-1\}} \mathbb{P}\big( \mathsf{G}_i \mid \mathcal{F}_{\eta}^{\pm} \big) \Big] \\
	  		\le &\mathbb{E}\Big[\mathbbm{1}_{\kappa\ge 1} \cdot  \prod\nolimits_{i\in \{1,-1\}} \Big( \sum\nolimits_{0\le j\le \kappa} \mathbb{P}\big(  \mathsf{G}_{i,j} \mid \mathcal{F}_{\eta}^{\pm} \big) \Big) \Big], 
	  	\end{split}
	  \end{equation}
	 	where $\mathsf{G}_{i,0}$ is the event that $A_i$ and $\bm{\mathcal{C}}(c^{i}r_i)$ are connected by $\cup(\mathcal{L}_{1/2}^{\mathbb{R}^d}- \mathfrak{L}[10r_1,\frac{r_{-1}}{10}])$, and for each $1\le j\le \kappa$, $\mathsf{G}_{i,j}$ is the event that $A_i$ is connected to $\mathrm{ran}(\eta^{i}_j)\cap \mathfrak{B}^{i}_{r_{i},c}$ by loops contained in $\mathfrak{B}^{i}_{r_{i},c}$ (we denote the point measure consisting of these loops by $\mathfrak{L}_i$). Since $\mathcal{L}_{1/2}^{\mathbb{R}^d}- \mathfrak{L}[10r_1,\frac{r_{-1}}{10}]\le \mathcal{L}_{1/2}^{\mathbb{R}^d}$, one has  
	 \begin{equation}\label{new217}
		\mathbb{P}\big(  \mathsf{G}_{i,0} \mid \mathcal{F}_{\eta}^{\pm}  \big)   \le \mathbb{P}\big(    A_i \xleftrightarrow{\cup  \mathcal{L}_{1/2}^{\mathbb{R}^d}}  \bm{\mathcal{C}}(c^i r_i) \big). 
	\end{equation}
	Meanwhile, for $1\le j\le \kappa$, it follows from Lemma \ref{lemma_domination} that 
\begin{equation}\label{new218}
	\begin{split}
			\mathbb{P}\big(  \mathsf{G}_{i,j} \mid \mathcal{F}_{\eta}^{\pm} \big)   \le  & \mathbb{P}\big(    A_i \xleftrightarrow{\cup ( \mathfrak{L}_i+\mathfrak{L}[c^{i/2}r_i,10c^{i/2}r_i])}  \bm{\mathcal{C}}(c^i r_i) \big) \\
			\le & \mathbb{P}\big(    A_i \xleftrightarrow{\cup  \mathcal{L}_{1/2}^{\mathbb{R}^d}}  \bm{\mathcal{C}}(c^i r_i) \big).
	\end{split}
\end{equation}
	Plugging (\ref{new217}) and (\ref{new218}) into (\ref{new216}), we obtain 
\begin{equation}\label{new219}
	\mathbb{P}\big(\mathsf{F} \big)\le  \mathbb{E}\big[ \mathbbm{1}_{\kappa\ge 1} \cdot(\kappa+1)^2 \big]  \cdot\prod\nolimits_{i\in \{1,-1\}}\mathbb{P}\big(    A_i \xleftrightarrow{\cup  \mathcal{L}_{1/2}^{\mathbb{R}^d}}  \bm{\mathcal{C}}(c^i r_i) \big).
\end{equation}	
	 
According to (\ref{ineq_exp_decay_crossing}), $\kappa$ decays exponentially with rate $O\big((r_1 / r_{-1})^{d-2}\big)$, which implies that $\mathbb{E}\big[ \mathbbm{1}_{\kappa\ge 1} \cdot(\kappa+1)^2 \big]\lesssim (r_1 / r_{-1})^{d-2}$. Combined with (\ref{new219}), it yields the desired bound (\ref{new215}). 
\end{proof}

    The second corollary is the quasi-submultiplicativity of connecting probabilities: 
  \begin{corollary}\label{coro_quasi_brownian}
	For any $d\ge 3$, there exist $c>c'>0$ such that for any $r>0$, $A_1 \subset \mathfrak{B}^{1}_{r,c'}$ and $A_{-1} \subset \mathfrak{B}^{-1}_{r,c'}$,
	\begin{equation}
		\mathbb{P}\big(    A_1 \xleftrightarrow{\cup  \mathcal{L}_{1/2}^{\mathbb{R}^d}   }  A_{-1} \big) \lesssim  \prod\nolimits_{i\in \{1,-1\}}	\mathbb{P}\big(    A_i \xleftrightarrow{\cup  \mathcal{L}_{1/2}^{\mathbb{R}^d}}  \bm{\mathcal{C}}(c^i r) \big).
	\end{equation}
\end{corollary}
   \begin{proof}
 The counterpart of this corollary for $\widetilde{\mathbb{Z}}^d$ is a special case of \cite[Lemma 5.2]{cai2024quasi}. Here we extend the proof therein to $\mathbb{R}^d$.

    We denote $\mathfrak{L}_1:=\mathfrak{L}[c^{\frac{2}{3}}r,c^{\frac{1}{3}}r]$ and $\mathfrak{L}_{-1}:=\mathfrak{L}[c^{-\frac{1}{3}}r,c^{-\frac{2}{3}}r]$. For each $i\in \{1,-1\}$, when $A_1 \xleftrightarrow{\cup  \mathcal{L}_{1/2}^{\mathbb{R}^d}  }  A_{-1}$ and $\mathfrak{L}_i= 0$ both occur, every cluster connecting $A_1$ and $A_{-1}$ must include a sub-cluster consisting of loops inside $\mathfrak{B}^{i}_{r,1}$ and intersecting both $A_i$ and $\bm{\mathcal{C}}(c^i r)$. In particular, $A_i$ is connected to $\bm{\mathcal{C}}(c^i r)$ by loops in $\mathcal{L}_{1/2}^{\mathbb{R}^d}\cdot \mathbbm{1}_{\mathrm{ran}(\ell)\subset \mathfrak{B}^{i}_{r,1}}$. Since $\mathfrak{B}^{1}_{r,1}$ and $\mathfrak{B}^{-1}_{r,1}$ are disjoint, these two connecting events (for $i=1$ and $i=-1$) are independent. Consequently, we have
     \begin{equation}\label{new221}
   	\begin{split}
   			\mathbb{P}\big(    A_1 \xleftrightarrow{\cup  \mathcal{L}_{1/2}^{\mathbb{R}^d}   }  A_{-1}, \mathfrak{L}_1=\mathfrak{L}_{-1}= 0 \big) \lesssim  \prod\nolimits_{i\in \{1,-1\}}	\mathbb{P}\big(    A_i \xleftrightarrow{\cup  \mathcal{L}_{1/2}^{\mathbb{R}^d}}  \bm{\mathcal{C}}(c^i r) \big).
   	\end{split}
   \end{equation} 
        Meanwhile, it follows from Corollary \ref{coro_with_crossing_loop} that for any $i\in \{1,-1\}$,  
 \begin{equation}\label{new222}
 \begin{split}
  	\mathbb{P}\big(    A_1 \xleftrightarrow{\cup  \mathcal{L}_{1/2}^{\mathbb{R}^d}   }  A_{-1}, \mathfrak{L}_i\neq  0 \big)  
 \lesssim     \prod\nolimits_{i\in \{1,-1\}}	\mathbb{P}\big(    A_i \xleftrightarrow{\cup  \mathcal{L}_{1/2}^{\mathbb{R}^d}}  \bm{\mathcal{C}}(c^i r_i) \big).   	
 \end{split}
 \end{equation}
   Combining (\ref{new221}) and (\ref{new222}), we complete the proof. 
\end{proof}

 \section{Crossing probability after the removal of small loops} \label{section_proof_crossing_probability}

This section is devoted to the proof of Theorem \ref{theorem_const_removal_edge_loop}. Let $C_\dagger>0$ be a sufficiently large constant that will be determined later. Without loss of generality, we assume $\widehat{N}:= N^{1\vee \frac{2}{d-4}}\ge C_\dagger^{10}n$ (otherwise, (\ref{goodineq_128}) and (\ref{goodineq_129}) follow directly from the trivial bound). For simplicity, we abbreviate $\mathcal{L}_{1/2}^{\widetilde{\mathbb{Z}}^d}$ as $\widetilde{\mathcal{L}}$, and abbreviate ``$\xleftrightarrow{\cup \widetilde{\mathcal{L}}}$'' as ``$\xleftrightarrow{}$''. Recall that $\mathfrak{L}$ is the point measure consisting of the loops in $\mathcal{L}_{1/2}^{\widetilde{\mathbb{Z}}^d}$ that are contained in a single interval $I_e$ for some $e\in \mathbb{L}^d$ and intersect both of its trisection points. For any $A,A'\subset \widetilde{\mathbb{Z}}^d$, we define $\mathsf{F}_{A,A'}$ as the event that $A$ and $A'$ are connected by $\cup \widetilde{\mathcal{L}}$ without having any pivotal loop in $\mathfrak{L}$. In particular, we have
\begin{equation}
	\big\{B(n) \xleftrightarrow{\cup (\widetilde{\mathcal{L}}-\mathfrak{L})} \partial B(N)\big\} \subset \mathsf{F}_{B(n), \partial B(N)}. 
\end{equation}
Combining this inclusion with (\ref{crossing_low}) and (\ref{crossing_high}), we obtain 
\begin{equation}
	\begin{split}
		&\mathbb{P}\big( B(n) \xleftrightarrow{\cup (\widetilde{\mathcal{L}}-\mathfrak{L})} \partial B(N) \big) \\
		 \lesssim & \big( n/\widehat{N} \big)^{(\frac{d}{2}-1)\vee  (d-4)} \cdot  \mathbb{P}\big(  \mathsf{F}_{B(n), \partial B(N)} \mid B(n) \xleftrightarrow{} \partial B(N) \big). 
	\end{split}
\end{equation}
Therefore, it is sufficient to prove that for some constant $c>0$,   
\begin{equation}\label{newfixadd_33}
	 \mathbb{P}\big(  \mathsf{F}_{B(n), \partial B(N)} \mid B(n) \xleftrightarrow{} \partial B(N) \big)\lesssim       \big( n  / \widehat{N} \big)^{c}. 
\end{equation} 
Before delving into the details, we first outline the heuristic leading to this estimate. It was shown in \cite[Lemma 5.3]{inpreparation_pivotal} that conditioned on $\{B(n) \xleftrightarrow{} \partial B(N)\}$, for any annulus
$B(M)\setminus B(m)$ satisfying $\widehat{N} \gg M\ge C m \gg n$, with uniformly positive probability each cluster connecting $B(n)$ and $\partial B(N)$ contains at least one pivotal loop in $\mathfrak{L}$. Pivotal loops contained in different annuli are in fact correlated, with correlations induced by loops crossing multiple annuli. Referring to \cite[Lemma 5.2]{inpreparation_pivotal}, the probability of such crossings decays sufficiently fast, which in turn suggests that pivotal loops in distant annuli are asymptotically independent. Note that there are of order $\mathrm{log}(\widehat{N}/n)$ such annuli, and that on the event $\mathsf{F}_{B(n), \partial B(N)}$, none of them contains a pivotal loop in $\mathfrak{L}$. Consequently, the conditional probability of $\mathsf{F}_{B(n),\partial B(N)}$ decays at least polynomially in $\widehat{N}/n$. In the remainder of this section, we present a detailed proof following this heuristic.

  For any $R\ge 1$, we denote by $\widetilde{B}(R):=\cup_{e\in \mathbb{L}^d:I_e \cap (-R,R)^d\neq \emptyset}I_e$ the box on $\widetilde{\mathbb{Z}}^d$ of radius $R$ centered at $\bm{0}$. For each $i\in \mathbb{N}$, let $r_i:= C_\dagger^i n$. Next, we recall the definition of the partial cluster $\mathfrak{C}_i$ from \cite[Section 3.2]{cai2024high}. Informally, $\mathfrak{C}_i$ is the loop cluster obtained by restricting the loop intersections to $\widetilde{B}(r_{2i})$. Precisely, for $i\in \mathbb{N}^+$, the partial cluster $\mathfrak{C}_i$ can be defined via the following inductive procedure: 
 \begin{itemize}

 	\item  Step $0$: we define $\mathcal{C}_0:=\widetilde{B}(n)$;

 	\item  Step $j$ ($j\ge 1$): Given $\mathcal{C}_{j-1}$, the cluster $\mathcal{C}_{j}$ is defined as the union of $\mathcal{C}_{j-1}$ and the ranges of all loops in $\widetilde{\mathcal{L}}$ intersecting $\mathcal{C}_{j-1}\cap \widetilde{B}(r_{2i})$. If no such loop exists (i.e., $\mathcal{C}_{j}=\mathcal{C}_{j-1}$), we stop the procedure and define $\mathfrak{C}_i:=\mathcal{C}_{j}$; otherwise, we proceed to Step $j+1$.

 \end{itemize}
 As noted in \cite[Section 5.2]{inpreparation_pivotal}, the cluster $\mathfrak{C}_i$ has  the following two properties:
\begin{enumerate}

	\item[(i)]  For any possible configuration $D$ of $\mathfrak{C}_i$, conditioned on $\{\mathfrak{C}_i=D \}$, the point measure consisting of all loops that are not included in $\mathfrak{C}_i$ shares the same distribution as $\widetilde{\mathcal{L}}\cdot \mathbbm{1}_{\mathrm{ran}(\ell)\cap D\cap \widetilde{B}(r_{2i})=\emptyset}$.


	\item[(ii)]  Arbitrarily take $i'>i\ge 1$ such that $n_{2i'}<N$. Given $\mathfrak{C}_i$, on the event $\{B(n)\xleftrightarrow{} \partial B(N)\}$, any loop in $\mathfrak{L}$ pivotal for $\{\mathfrak{C}_i \xleftrightarrow{} \partial B(n_{2i'}) \}$ is also pivotal for $\{B(n)\xleftrightarrow{} \partial B(N)\}$. In addition, all such loops are measurable with respect to $\mathcal{F}_{\mathfrak{C}_j}$ for all $j\ge i'$. Here $\mathcal{F}_{\mathcal{C}}$ denotes the $\sigma$-field generated by the cluster $\mathcal{C}$ and all loops included in $\mathcal{C}$.


\end{enumerate}
For convenience, we set $\mathfrak{C}_0:=\widetilde{B}(n)$. In particular, $\mathcal{F}_{\mathfrak{C}_0}$ is trivial.

We define $i_\star= i_\star(n,N):= \min \{ i\ge 1: r_{2i+3}\ge \widehat{N}    \}$. Note that $i_\star \asymp \mathrm{ln} ( \widehat{N}/n )$. It follows from Property (ii) that on $\mathsf{F}_{B(n), \partial B(N)}$, for any $1\le i<i' \le i_\star$, no loop in $\mathfrak{L}$ is pivotal for $\{\mathfrak{C}_i \xleftrightarrow{} \partial B(n_{2i'}) \}$ (we denote this event by $\mathsf{G}_{i,i'}$). Therefore, 
 \begin{equation}\label{new3.3}
 	\mathbb{P}\big( \mathsf{F}_{B(n), \partial B(N)} \big)\le  \mathbb{P}\big( B(n)\xleftrightarrow{} \partial B(N), \cap_{1\le i<i' \le i_\star} \mathsf{G}_{i,i'}\big).  
 \end{equation}


 

For each $i\in \mathbb{N}^+$, we define $i_+:=\min \{j\ge i+1: \mathfrak{C}_i\cap \partial B(r_{2j-1})= \emptyset \}$. We also set $0_+:=1$. For any $j\in [ i_++2,i_{\star}]$, if $(i_+)_+\ge j$ (i.e., $\mathfrak{C}_{i_+}$ intersects $\partial B(r_{2j-3})$), then the cluster $\mathfrak{C}_{i_+}$ must contain a loop that intersects both $B(r_{2i_+})$ and $\partial B(r_{2j-3})$ (we denote the collection of such loops by $\widetilde{\mathfrak{L}}[r_{2i_+}, r_{2j-3}]$). As a result,
\begin{equation}\label{24}
	\begin{split}
		& \mathbb{P}\big( B(n)\xleftrightarrow{} \partial B(N), (i_+)_+\ge j  \mid \mathcal{F}_{\mathfrak{C}_i} \big) \\
		\overset{ }{\le} &    \mathbb{P}\big( \mathfrak{C}_i \xleftrightarrow{} \partial B(N), \widetilde{\mathfrak{L}}[r_{2i_+}, r_{2j-3}] \neq 0  \mid \mathcal{F}_{\mathfrak{C}_i} \big)\\
		\lesssim  &   \sqrt{\tfrac{r_{2i_+}}{r_{2j-3}}} \cdot \mathbb{P}\big(B(n)\xleftrightarrow{} \partial B(N)   \mid \mathcal{F}_{\mathfrak{C}_i} \big),
	\end{split}
\end{equation} 
where the last inequality follows from Property (i) and \cite[Lemma 5.2]{inpreparation_pivotal}.


We define an increasing sequence $\{i_k\}_{0\le k \le k_{\star}}$ recursively as follows: (i) set $i_0:=0$ and $i_1:=1$; (ii) for $k\ge 1$, if $i_k\ge i_\star$, we stop the construction and set $k_{\star}:=k$; otherwise, we set $i_{k+1}:=(i_k)_+$ and proceed to the next step (i.e., check whether $i_{k+1} \ge i_\star$ and continue accordingly). The next lemma shows that with high probability, this construction runs for sufficiently many steps before the sequence reaches $i_\star$. We denote $K:=\mathrm{ln}(\widehat{N}/n)$, and let $\mathsf{H}_{j}:=\{k_{\star} \le j\}$ for $j\ge 1$.

\begin{lemma}
	There exist constants $c_{\spadesuit},c_{\clubsuit}>0$ such that 
	\begin{equation}\label{new35}
		\mathbb{P}\big(\mathsf{H}_{c_{\spadesuit}K}  \mid B(n)\xleftrightarrow{} \partial B(N)  \big) \lesssim \big( n / \widehat{N} \big)^{c_{\clubsuit}}.
\end{equation}
\end{lemma}
\begin{proof}
	We divide the event $\mathsf{H}_{c_{\spadesuit}K}$ into two sub-events 
\begin{equation}
	\mathsf{H}_{c_{\spadesuit}K}^{(1)}:= \mathsf{H}_{c_{\spadesuit}K}\cap \big\{i_{k_\star-1}\le \tfrac{1}{2}i_{\star}  \big\}\ \ \text{and}\ \ \mathsf{H}_{c_{\spadesuit}K}^{(2)}:= \mathsf{H}_{c_{\spadesuit}K}\cap \big\{i_{k_\star-1}> \tfrac{1}{2}i_{\star} \big\}, 
\end{equation}
and estimate their conditional probabilities separately.


 For $\mathsf{H}_{c_{\spadesuit}K}^{(1)}$, note that for any $0\le i<i'\le \frac{1}{2}i_{\star}$, the event $\{i_+= i' \}$ is measurable with respect to $\mathcal{F}_{\mathfrak{C}_{i}}$. Therefore, we have
 \begin{equation}
 	\begin{split}
 		&	\mathbb{P}\big( B(n)\xleftrightarrow{} \partial B(N) ,\mathsf{H}_{c_{\spadesuit}K}^{(1)} , i_{k_\star-2} =i,  i_{k_\star-1} =i' \big)  \\
 		\le  & \mathbb{E}\big[ \mathbbm{1}_{i_+= i' } \cdot \mathbb{P}\big( B(n)\xleftrightarrow{} \partial B(N) , (i_+)_+\ge  i_{\star} \mid \mathcal{F}_{\mathfrak{C}_{i}} \big)   \big] \\
 		\overset{(\ref{24})}{\lesssim } &  \sqrt{\tfrac{r_{2i'}}{r_{2i_{\star}-3}}} \cdot \mathbb{P}\big( B(n)\xleftrightarrow{} \partial B(N) \big) \lesssim e^{c(i'-i_{\star})}\cdot \mathbb{P}\big( B(n)\xleftrightarrow{} \partial B(N) \big). 
 	\end{split}
 \end{equation}
Summing over all $0\le i<i'\le \frac{1}{2}i_{\star}$, this yields 
\begin{equation}\label{new38}
	\mathbb{P}\big(\mathsf{H}_{c_{\spadesuit}K}^{(1)}  \mid B(n)\xleftrightarrow{} \partial B(N)  \big) \lesssim e^{-c'i_{\star}}.
	\end{equation}

 For  $\mathsf{H}_{c_{\spadesuit}K}^{(2)}$, it follows from (\ref{24}) that for each $1\le k\le k_\star-2$, given $\{i_j\}_{0\le j\le k}$, the increment $i_{k+1}-i_k$ decays exponentially with rate $O(C_\dagger^{-1})$. Therefore, on the event $\mathsf{H}_{c_{\spadesuit}K}$, $i_{k_\star-1}$ is stochastically dominated by a negative binomial random variable $\mathbf{Y}$, which counts the number of trials required to obtain $c_{\spadesuit}K$ successes, with success probability $O(C_\dagger^{-1})$ in each trial. Thus, by taking a sufficiently large $C_\dagger$, we have
 \begin{equation}\label{new39}
 		\mathbb{P}\big(\mathsf{H}_{c_{\spadesuit}K}^{(2)}  \mid B(n)\xleftrightarrow{} \partial B(N)  \big) \le \mathbb{P}(\mathbf{Y}>\tfrac{1}{2}i_{\star} ) \lesssim e^{-c''i_{\star}}.
 \end{equation}
 Combining (\ref{new38}), (\ref{new39}) and $i_\star \asymp \mathrm{ln} ( \widehat{N}/n )$, we obtain this lemma.
\end{proof}

 In what follows, we aim to show that for some $c_{\triangle}>0$, 
  \begin{equation}\label{new310}
  	\mathbb{P}\big(  (\cap_{1\le i<i'\le i_\star} \mathsf{G}_{i,i'}) \cap  \mathsf{H}_{c_{\spadesuit}K}^c \mid B(n)\xleftrightarrow{} \partial B(N)  \big) \lesssim \big( n / \widehat{N} \big)^{c_{\triangle}}. 
  \end{equation}
To see this, note that 
 \begin{equation}\label{new311}
 	  (\cap_{1\le i<i'\le i_\star} \mathsf{G}_{i,i'}) \cap  \mathsf{H}_{c_{\spadesuit}K}^c   \subset    (\cap_{1\le k\le \lfloor  c_{\spadesuit}K \rfloor -2 } \mathsf{G}_{i_k, i_{k+1}})\cap  \{ i_{ \lfloor  c_{\spadesuit}K \rfloor-1}< i_{\star}     \}. 
 \end{equation}
 Moreover, for each $m\ge 3$, since $\cap_{1\le k\le m-2} \mathsf{G}_{i_k,i_{k+1}}$ and $i_{m}$ are both measurable with respect to $\mathcal{F}_{\mathfrak{C}_{i_{m-1}}}$ (by Property (ii)), one has that for some constant $c_{\diamondsuit}\in (0,1)$,  
  \begin{equation}\label{new312}
 \begin{split}
 	&  	 \mathbb{P}\big(B(n)\xleftrightarrow{} \partial B(N), \cap_{1\le k\le m-1} \mathsf{G}_{i_k,i_{k+1}}, i_{m} < i_{\star}   \big) \\
 	 	 = & \mathbb{E}\big[ \mathbbm{1}_{\cap_{1\le k\le m-2} \mathsf{G}_{i_k,i_{k+1}}\cap \{i_m< i_{\star} \}} \cdot \mathbb{P}\big(B(n)\xleftrightarrow{} \partial B(N),  \mathsf{G}_{i_{m-1},i_m} \mid \mathcal{F}_{\mathfrak{C}_{i_{m-1}}} \big) \big]  \\
 	 	 \le & (1-c_{\diamondsuit})\cdot  \mathbb{P}\big(B(n)\xleftrightarrow{} \partial B(N),   \cap_{1\le k\le m-2 } \mathsf{G}_{i_k,i_{k+1}}, i_{m-1} < i_{\star}  \big),  
 \end{split}
 \end{equation}
 where in the last line we used $\{i_m <  i_{\star} \}\subset \{i_{m-1} < i_{\star} \}$ and \cite[Lemma 5.3]{inpreparation_pivotal}. By applying (\ref{new312}) repeatedly, we get 
 \begin{equation}
 	\mathbb{P}\big( \cap_{1\le k\le \lfloor  c_{\spadesuit}K \rfloor -2 } \mathsf{G}_{i_k, i_{k+1}}, i_{ \lfloor  c_{\spadesuit}K \rfloor-1}< i_{\star}   \mid B(n)\xleftrightarrow{} \partial B(N)  \big)\le (1-c_{\diamondsuit})^{\lfloor  c_{\spadesuit}K \rfloor-2}. 
 \end{equation}  
Combined with (\ref{new311}), it implies (\ref{new310}).

Putting (\ref{new3.3}), (\ref{new35}) and (\ref{new310}) together, we establish the desired bound (\ref{newfixadd_33}), which completes the proof of Theorem \ref{theorem_const_removal_edge_loop}.  \qed



 \section{One-arm exponent and dimension of Brownian loop clusters}\label{section_existence_one_arm}
 
 In this section, we present the proof of Theorem \ref{theorem_main}, which includes the following three parts: 
\begin{itemize}

	\item  Section \ref{subsection_exist_zeta}: We adapt the proof of Fekete's subadditive lemma to establish the existence of the one-arm exponent $\zeta$.

	\item Section \ref{section_zeta_larger_than}: By comparing the one-arm probability of $\mathcal{L}_{1/2}^{\mathbb{R}^3}$ with that of $\mathcal{L}_{1/2}^{\widetilde{\mathbb{Z}}^3}$ (via the coupling in Lemma \ref{lemma_coupling}), we show that $\zeta>\frac{1}{2}$.

	\item  Section \ref{section_cluster_dimension}: We use the first moment method to prove that almost surely the upper box-counting dimension of every loop cluster of $\mathcal{L}_{1/2}^{\mathbb{R}^3}$ is bounded from above by $3-\zeta$.

\end{itemize}

 \subsection{Proof of the existence of $\zeta$}\label{subsection_exist_zeta}

 We denote $q(\epsilon):= 	\mathbb{P}\big( \bm{\mathcal{C}}(\epsilon) \xleftrightarrow{\cup \mathcal{L}_{1/2}^{\mathbb{R}^3}} \bm{\mathcal{C}}(1) \big)$. The scaling invariance of $\mathcal{L}_{1/2}^{\mathbb{R}^3}$ implies that 
\begin{equation}\label{new41}
	\mathbb{P}\big( \bm{\mathcal{C}}(\epsilon r) \xleftrightarrow{\cup \mathcal{L}_{1/2}^{\mathbb{R}^3}} \bm{\mathcal{C}}(r) \big) = q(\epsilon), \  \forall r>0. 
\end{equation} 
By Corollary \ref{coro_quasi_brownian}, there exist $C_\dagger,C_\ddagger>0$ such that for any $\epsilon_1,\epsilon_2\in (0,C_\ddagger^{-1})$, 
\begin{equation}\label{new43}
	\begin{split}
			q(\epsilon_1\epsilon_2) \le& C_\dagger\cdot   \mathbb{P}\big( \bm{\mathcal{C}}(\epsilon_1\epsilon_2) \xleftrightarrow{\cup \mathcal{L}_{1/2}^{\mathbb{R}^3}} \bm{\mathcal{C}}(C_\ddagger^{-1}\epsilon_2) \big)\cdot \mathbb{P}\big( \bm{\mathcal{C}}(C_\ddagger\epsilon_2) \xleftrightarrow{\cup \mathcal{L}_{1/2}^{\mathbb{R}^3}} \bm{\mathcal{C}}(1) \big)\\
			\overset{}{=} &  C_\dagger\cdot q(C_\ddagger\epsilon_1) \cdot q(C_\ddagger\epsilon_2).
	\end{split}
\end{equation}
For $k\in \mathbb{N}$, let $u_{k}:=\mathrm{ln}\big(q(C_\ddagger^{-k})\big)$. Note that $\{ u_{k} \}_{k\in \mathbb{N}}$ is a decreasing sequence since $q(\epsilon)$ is increasing in $\epsilon$. It follows from (\ref{new43}) that for any $k_1,k_2\ge 1$,  
\begin{equation}\label{new44}
	u_{k_1+k_2}\le u_{k_1-1}+u_{k_2-1}+ C_{ \star}, 
\end{equation}
where $C_{ \star}:=\mathrm{ln}(C_\dagger)$. Arbitrarily fix $k\in \mathbb{N}^+$. Then any integer $n\in \mathbb{N}$ can be written as $(k+2)m+r$ for some $m=m(k,n)\in \mathbb{N}$ and $r=r(k,n)\in \{0,1,...,k+1\}$. Thus, by applying (\ref{new44}) repeatedly, we have
\begin{equation}
	\begin{split}
		u_{n} \le u_{(k+2)m} \overset{}{ \le}  u_{(k+2)(m-1)} + u_{k}+  C_{ \star} \overset{}{ \le} ... \overset{}{ \le} m( u_{k}+  C_{ \star}).
	\end{split}
\end{equation}
By dividing both sides by $n$ and then taking the upper limit as $n\to \infty$, we obtain  
\begin{equation}\label{new46}
	\begin{split}
		\limsup_{n\to \infty}\frac{u_n}{n}  \le  \limsup_{n \to \infty} \frac{m (u_k + C_{ \star})}{ n}= \frac{u_k + C_{ \star}}{k+2}. 
	\end{split}
\end{equation} 
 Note that (\ref{newadd210}) implies $u_{k}\ge \mathrm{ln}\big( c\cdot C_\ddagger^{-k}  \big)= - \mathrm{ln}(C_\ddagger)\cdot k+\mathrm{ln}(c)$. Therefore, 
  \begin{equation*}
	\begin{split}
		\liminf_{k\to \infty }\frac{u_k + C_{ \star}}{k+2} =&  	\liminf_{k\to \infty } \big(\frac{u_k }{k} + \frac{C_{ \star}}{k+2}- \frac{2u_k}{k(k+2)}  \big) \\
		\le & 	\liminf_{k\to \infty } \big(\frac{u_k }{k} + \frac{C_{ \star}}{k+2}+ \frac{2\mathrm{ln}(C_\ddagger)\cdot k - 2\mathrm{ln}(c)}{k(k+2)}  \big) = \liminf_{k\to \infty }\frac{u_k}{k}.  
	\end{split}
\end{equation*}
Combined with (\ref{new46}), it yields that the limit $a:=\lim_{k\to \infty}\frac{u_k}{k}$ exists. As a result, 
\begin{equation}
	q(C_\ddagger^{-k}) = e^{ak+o(k)}, \  \forall k\ge 1.  
\end{equation}
This together with the monotonicity of $q(\epsilon)$ immediately implies the desired bound (\ref{ineq_main1}) with $\zeta=-a\cdot [\mathrm{ln}(C_{\ddagger})]^{-1}$.   \qed





\begin{remark}
	The proof in this subsection relies only on the scaling invariance of the Brownian loop soup and the decomposition in Corollary \ref{coro_quasi_brownian}, both of which hold for all 	$\alpha>0$. Therefore, the one-arm exponent of $\mathcal{L}_{\alpha}^{\mathbb{R}^3}$ exists for every $\alpha>0$. In addition, the argument in Remark \ref{remark_zeta_<1} also applies to all $\alpha>0$, which shows that this exponent is strictly less than $1$.  
\end{remark}

\subsection{Proof of $\zeta>\frac{1}{2}$} \label{section_zeta_larger_than}

For any $\delta>0$, we denote by $\mathcal{L}^{\ge \delta}$ the point measure consisting of loops in $\mathcal{L}_{1/2}^{\mathbb{R}^3}$ with diameter at least $\delta$. Since $\mathsf{F}_\epsilon:= \big\{ \bm{\mathcal{C}}(\epsilon) \xleftrightarrow{\cup \mathcal{L}_{1/2}^{\mathbb{R}^3}} \bm{\mathcal{C}}(1) \big\}$ is certified by finite many loops, we have 
\begin{equation}
	  \mathbbm{1}_{\mathsf{F}_\epsilon^{  \ge \delta}} := \mathbbm{1}_{ \bm{\mathcal{C}}(\epsilon) \xleftrightarrow{\cup \mathcal{L}^{\ge \delta}} \bm{\mathcal{C}}(1)} \overset{\mathrm{a.s.}}{\longrightarrow } \mathbbm{1}_{\mathsf{F}_\epsilon} \ \ \text{as}\  \delta \downarrow 0. 
\end{equation}
Therefore, since $\mathsf{F}_\epsilon^{ \ge \delta}$ is decreasing in $\delta$, it follows from the monotone convergence theorem that  
 \begin{equation}\label{new49}
 	  \mathbb{P}\big(\mathsf{F}_\epsilon  \big) = \lim\limits_{ \delta \downarrow 0} \mathbb{P}\big(\mathsf{F}_\epsilon^{\ge \delta}  \big). 
 \end{equation}
Thus, to obtain $\zeta>\frac{1}{2}$, it suffices to show that for some constant $c_\dagger>0$, 
\begin{equation}\label{add410}
	\mathbb{P}\big(\mathsf{F}_\epsilon^{\ge \delta}  \big) \lesssim \epsilon^{\frac{1}{2}+c_\dagger}, \ \ \forall \epsilon\in (0,1)\ \text{and}\  \delta\in (0, 
	  e^{-\epsilon^{-1}}). 
\end{equation}
 The remainder of this subsection is devoted to proving (\ref{add410}).

Without loss of generality, we assume that $\epsilon>0$ is sufficiently small. Arbitrarily fix $\delta\in (0, e^{-\epsilon^{-1}})$. Next, we approximate the event $\mathsf{F}_\epsilon^{\ge \delta}$ step by step, and control the error at each step. We define the event 
\begin{equation}
	\mathsf{G}:=  \big\{ \exists \ell \in  \mathcal{L}_{1/2}^{\mathbb{R}^3} \  \text{intersecting}\ \bm{\mathcal{C}}(1)\ \text{and}\ \bm{\mathcal{C}}(\epsilon^{-1})  \big\},  
\end{equation}
and let $\mathsf{F}^{(1)}:= \mathsf{F}_\epsilon^{\ge \delta}\cap  \mathsf{G}^c$. By the union bound and Lemma \ref{lemma_crossing_loop_mass}, we have 
\begin{equation}\label{show412}
	\mathbb{P}\big( \mathsf{F}_\epsilon^{\ge \delta}\big) \le \mathbb{P}\big( \mathsf{F}^{(1)} \big)+ \mathbb{P}\big( \mathsf{G} \big)\le   \mathbb{P}\big( \mathsf{F}^{(1)} \big)+ C\epsilon. 
\end{equation}
 Note that $\mathsf{F}_\epsilon^{\ge \delta}$ depends only on the loops  intersecting $\bm{\mathcal{B}}(1)$. In addition, the absence of $\mathsf{G}$ ensures that all these loops are contained in $\bm{\mathcal{B}}(\epsilon^{-1})$. Hence, the event $\mathsf{F}^{(1)}$ relies only on the loops contained in $\bm{\mathcal{B}}(\epsilon^{-1})$ with diameter at least $\delta$ (we denote the collection of these loops by $\mathcal{A}_*$). Subsequently, we approximate the loops in $\mathcal{A}_*$ by the discrete loops in $\mathcal{L}_{1/2}^{\mathbb{Z}^3}$. Recall the collections $\mathcal{A}_{\alpha,r,\theta,N}^{\mathbb{Z}^d}$ and $\mathcal{A}_{\alpha,r,\theta,N}^{\mathbb{R}^d}$ defined in (\ref{mathcalA_Zd}) and (\ref{mathcalA_Rd}), and the bijection $\varpi_{\alpha,r,\theta,N}$ given by Lemma \ref{lemma_coupling}. In this proof, we take $d=3$, $\alpha=1/2$, $\theta=1$, $r=2\epsilon^{-1}$ and $N=\lfloor e^{2\delta^{-1}}  \rfloor$. With this choice of parameters, we write $\overline{\mathcal{A}}:=\mathcal{A}_{\alpha,r,\theta,N}^{\mathbb{Z}^d}$, $\widetilde{\mathcal{A}}:=\mathcal{A}_{\alpha,r,\theta,N}^{\mathbb{R}^d}$ and $\varpi:=\varpi_{\alpha,r,\theta,N}$. By Lemma \ref{lemma_coupling}, there exists a constant $C_\ddagger>0$ such that  
 \begin{equation}\label{add413}
 	\begin{split}
 		 \mathbb{P}(\mathsf{A}^c) \le C_\ddagger(\alpha+1) r^3 N^{-\frac{1}{2}}\asymp \epsilon^{-3} e^{-\delta^{-1}} \lesssim \epsilon, 
 	\end{split}
 \end{equation}
 where $\mathsf{A}$ is the event that there exists a bijection $\varpi: \widetilde{\mathcal{A}} \to \overline{\mathcal{A}}$ satisfying 
 \begin{equation}\label{reviseaddto_415}
 		\mathrm{d}_{\mathrm{F}}\big(\ell,\varpi(\ell)\big) \le C_\ddagger N^{-\frac{1}{4}}\mathrm{ln}(N), \   \forall \ell\in \widetilde{\mathcal{A}}. 
 \end{equation}
 Note that each loop $\ell \in \mathcal{A}_* \setminus \widetilde{\mathcal{A}}$ has diameter at least $\delta \asymp [\mathrm{ln}(N)]^{-1}$ and satisfies $T_{\ell} \le N^{-1}$. In fact, with high probability such a loop is absent. Precisely, the ball $\bm{\mathcal{B}}(\epsilon^{-1})$ can be covered by $O((\epsilon \delta)^{-3})$ boxes of radius $\tfrac{\delta}{100}$. In addition, for each of these boxes, the expected number of loops intersecting it with diameter at least $\delta$ is of order $1$ (by Lemma \ref{lemma_crossing_loop_mass}); moreover, for any loop $\ell$ among them, if it satisfies $T_{\ell}\le N^{-1}$, then it contains a trajectory of Brownian motion with duration at most $N^{-1}$ that reaches a distance of $c[\mathrm{ln}(N)]^{-1}$, which occurs with probability at most $Ce^{-c'N [\mathrm{ln}(N)]^{-2}}$ (see e.g., \cite[Remark 2.22]{morters2010brownian}). To sum up, we obtain 
 \begin{equation}\label{add414}
	\mathbb{P}\big(  \mathcal{A}_* \not\subset  \widetilde{\mathcal{A}} \big) \lesssim (\epsilon \delta)^{-3}\cdot e^{-c'N [\mathrm{ln}(N)]^{-2}}\lesssim \epsilon. 
\end{equation}
We define the event $\mathsf{F}^{(2)}:= \mathsf{F}^{(1)}\cap \mathsf{A} \cap \{ \mathcal{A}_*  \subset  \widetilde{\mathcal{A}}\}$. By (\ref{add413}) and (\ref{add414}), we have 
\begin{equation}\label{show415}
	\begin{split}
		\mathbb{P}\big( \mathsf{F}^{(1)} \big)\le  \mathbb{P}\big( \mathsf{F}^{(2)} \big) + \mathbb{P}\big( \mathsf{A}^c \big) +  	\mathbb{P}\big(  \mathcal{A}_* \not\subset  \widetilde{\mathcal{A}}  \big)  \le \mathbb{P}\big( \mathsf{F}^{(2)} \big) +C\epsilon. 
	\end{split}
\end{equation}

We now estimate the probability of $\mathsf{F}^{(2)}$. When the event $\mathsf{F}^{(1)}$ occurs, there exists a sequence $\{\ell_i\}_{1\le i\le k}$ of loops in $\mathcal{A}_*$ such that $\mathrm{ran}(\ell_1)\cap  \bm{\mathcal{C}}(\epsilon)\neq \emptyset$, $\mathrm{ran}(\ell_k)\cap  \bm{\mathcal{C}}(1)\neq \emptyset$ and $\mathrm{ran}(\ell_i)\cap  \mathrm{ran}(\ell_{i+1})\neq \emptyset$ for all $1\le i\le k-1$. As a result, on $\mathsf{F}^{(2)}$, either the union of the ranges of  loops $\{\varpi(\ell_i)\}_{1\le i\le k}$ connects $\bm{\mathcal{C}}(2\epsilon)$ and $\bm{\mathcal{C}}(\frac{1}{2})$, or there exists $1\le i\le k-1$ such that $\varpi(\ell_i)$ and $\varpi(\ell_{i+1})$ are disjoint. For simplicity, we denote the Euclidean distance between the ranges of two curves $\eta,\eta'$ in $\mathbb{R}^3$ by $\|\eta- \eta'\|$. For each $1\le i\le k-1$, since $\mathrm{ran}(\ell_i)\cap  \mathrm{ran}(\ell_{i+1})\neq \emptyset$ (i.e., $\|\ell_i-\ell_{i+1}\|=0$), 
 \begin{equation}
 	\|\varpi(\ell_i) - \varpi(\ell_{i+1}) \| \le \|\ell_i - \varpi(\ell_i)  \| + \|\ell_{i+1} - \varpi(\ell_{i+1})  \| \overset{(\ref{reviseaddto_415})}{\le } 2C_{\ddagger} N^{-\frac{1}{4}}\mathrm{ln}(N). 
 \end{equation} 
  Thus, since each $\varpi(\ell_i)$ is contained in $\bm{\mathcal{B}}(4\epsilon^{-1})$ and has diameter at least $\frac{\delta}{2}$,   
\begin{equation}\label{show417}
	\mathbb{P}\big( \mathsf{F}^{(2)} \big) \le \mathbb{P}\big( \bm{\mathcal{C}}(\epsilon) \xleftrightarrow{N^{-1} \cdot \cup \mathcal{L}_{1/2}^{\mathbb{Z}^3}} \bm{\mathcal{C}}(1)  \big)  +  \mathbb{P} ( \mathsf{H}  ), 
\end{equation}
where $\mathsf{H}$ is the event that there exist two disjoint loops in $N^{-1}\cdot \mathcal{L}_{1/2}^{\mathbb{Z}^3}$ such that they are contained in $\bm{\mathcal{B}}(4\epsilon^{-1})$ and have diameter at least $\frac{\delta}{2}$, and their Euclidean distance is at most $2C_{\ddagger}N^{-\frac{1}{4}}\mathrm{ln}(N)$. Referring to Remark \ref{remark1.7_Zd}, one has 
\begin{equation}\label{show418}
	\begin{split}
		  \mathbb{P}\big( \bm{\mathcal{C}}(\epsilon) \xleftrightarrow{N^{-1} \cdot \cup \mathcal{L}_{1/2}^{\mathbb{Z}^3}} \bm{\mathcal{C}}(1)  \big)  
		\le  \mathbb{P}\big( B(10\epsilon N) \xleftrightarrow{ \cup \mathcal{L}_{1/2}^{\mathbb{Z}^3}} \partial B(\tfrac{N}{10})  \big) \lesssim  \epsilon^{\frac{1}{2}+c_\dagger}  
	\end{split}
\end{equation}
for some constant $c_\dagger>0$. We now estimate $\mathbb{P}(\mathsf{H})$. By a standard covering lemma, there exists a collection $\mathfrak{D}$ of $O\big((\epsilon \chi)^{-3}\big)$ boxes of radius $\chi:=\frac{10C_\ddagger \mathrm{ln}(N)}{N^{1/4}}$ that covers $\bm{\mathcal{B}}(4\epsilon^{-1})$. Therefore, on the event $\mathsf{H}$, there exists $\Lambda_z(\chi)\in \mathfrak{D}$ such that the collection $\mathfrak{L}_z:= \big\{ \widehat{\ell}\in N^{-1}\cdot \mathcal{L}_{1/2}^{\mathbb{Z}^3}:  \widehat{\ell}\ \text{intersects}\  \Lambda_z(\chi)\ \text{and}\ \bm{\mathcal{C}}_z(\frac{\delta}{4}) \big\}$ contains two disjoint loops (we denote this event by $\mathsf{H}_{\Lambda_z(\chi)}$). Thus, by applying the union bound, we have 
\begin{equation}\label{nice419}
	\mathbb{P} ( \mathsf{H}  ) \le \sum\nolimits_{\Lambda_z(\chi)\in \mathfrak{D} } \mathbb{P}\big( \mathsf{H}_{\Lambda_z(\chi)} \big) \lesssim (\epsilon \chi)^{-3} \cdot  \max_{\Lambda_z(\chi)\in \mathfrak{D}}\mathbb{P}(\mathsf{H}_{\Lambda_z(\chi)} ).  
\end{equation}

We present the following lemma to bound the right-hand side of (\ref{nice419}). 
 \begin{lemma}\label{newlemma41}
There exists $c_{\diamond}>0$ such that for any $z\in \mathbb{R}^3$, 
 	\begin{equation}\label{ineq_newlemma41}
 		\mathbb{P}\big(  \mathsf{H}_{\Lambda_z(\chi)}  \big) \lesssim  \big( \chi/ \delta  \big)^{3+c_{\diamond}}. 
 	\end{equation}
 \end{lemma}

 By plugging (\ref{ineq_newlemma41}) into (\ref{nice419}), one has 
 \begin{equation}\label{show421}
	\mathbb{P}\big( \mathsf{H} \big) \lesssim   \epsilon^{-3}\delta^{-3-c_\diamond}\chi^{c_\diamond}\lesssim \epsilon.  
\end{equation}
Putting (\ref{show417}), (\ref{show418}) and (\ref{show421}) together, we obtain $\mathbb{P}\big( \mathsf{F}^{(2)} \big) \lesssim \epsilon^{\frac{1}{2}+c_\dagger}$. Combined with (\ref{show412}) and (\ref{show415}), it implies (\ref{add410}). In conclusion, we have confirmed $\zeta>\frac{1}{2}$ assuming Lemma \ref{newlemma41}. We now verify Lemma \ref{newlemma41} to complete the proof.




\begin{proof}[Proof of Lemma \ref{newlemma41}]
 On $\mathsf{H}_{\Lambda_z(\chi)}\cap \{ |\mathfrak{L}_z|\ge 3\}$, there exist two loops $\ell$ and $\ell'$ in $\mathfrak{L}_z$ that certify $\mathsf{H}_{\Lambda_z(\chi)}$ (i.e., $\mathrm{ran}(\ell)\cap \mathrm{ran}(\ell')=\emptyset$); in addition, any other loop in $\mathfrak{L}_z$ certifies $\{\mathfrak{L}_z\neq \emptyset \}$. Thus, by the BKR inequality (see e.g., \cite[Lemma 3.3]{cai2024high}),  
	\begin{equation}\label{fix422}
		\begin{split}
			\mathbb{P}\big( \mathsf{H}_{\Lambda_z(\chi)}\cap \{ |\mathfrak{L}_z|\ge 3\} \big)  \le \mathbb{P}\big( \mathsf{H}_{\Lambda_z(\chi)}  \big)  \cdot \mathbb{P}\big( \mathfrak{L}_z\neq \emptyset \big). 
		\end{split}
	\end{equation}
	This together with the inclusion $\mathsf{H}_{\Lambda_z(\chi)}\subset \{|\mathfrak{L}_z|\ge 2\}$ implies that 
	\begin{equation}\label{newfix_422}
		\mathbb{P}\big( \mathsf{H}_{\Lambda_z(\chi)}  \big) \le \big[ 1- \mathbb{P}\big( \mathfrak{L}_z\neq \emptyset \big)\big]^{-1}\cdot \mathbb{P}\big( \widehat{\mathsf{H}} \big), 
	\end{equation}
	where $\widehat{\mathsf{H}}:= \mathsf{H}_{\Lambda_z(\chi)}\cap \{ |\mathfrak{L}_z|=  2\} $. According to \cite[Lemma 2.7]{chang2016phase}, one has 
	 \begin{equation}\label{good422}
	 	\mu^{\mathbb{Z}^3}\big( \big\{\ell : \ell\ \text{intersects both}\ \Lambda_{N\cdot z}(\chi N)\ \text{and}\ \bm{\mathcal{C}}_{N\cdot z}(\tfrac{\delta N}{4}) \big\} \big) \lesssim \chi/ \delta.  
	 \end{equation}
		Therefore, $\mathbb{P} ( \mathfrak{L}_z\neq \emptyset  )  \le \frac{C\chi}{\delta}<\frac{1}{2}$. Combined with (\ref{newfix_422}), it yields that  
		 \begin{equation}\label{fix_424}
		 	\mathbb{P}\big( \mathsf{H}_{\Lambda_z(\chi)}  \big) \le 2\cdot \mathbb{P}\big(\widehat{\mathsf{H}} \big). 
		 \end{equation}


		 In what follows, we estimate the probability of $\widehat{\mathsf{H}}$. When $\{|\mathfrak{L}_z|=  2\}$ occurs, we denote the loops in $\mathfrak{L}_z$ by $\ell_1$ and $\ell_2$. Hence, $\widehat{\mathsf{H}}$ can be equivalently written as $\{ |\mathfrak{L}_z|= 2, \mathrm{ran}(\ell_1)\cap \mathrm{ran}(\ell_2)=\emptyset\}$. Note that $\ell_1$ contains a trajectory $\eta$ of simple random walk on $N^{-1}\cdot \mathbb{Z}^3$ starting from $\Lambda_{z}(\chi )$ and stopped upon hitting $\bm{\mathcal{C}}_{z}(\tfrac{\delta}{4})$. By \cite[Lemma 2.5]{cutpointlemma}, there exists $c_{\star}>0$ such that 
	 \begin{equation}\label{show423}
	 	\mathbb{P}_{\eta}\Big( \max_{x\in \Lambda_{N\cdot  z}(\chi N)\cap  \mathbb{Z}^3} \mathbb{P}_{x}^{\mathbb{Z}^3}\big( \tau_{\bm{\mathcal{C}}_{N\cdot z}(\frac{\delta N}{4})} <\tau_{ N\cdot \mathrm{ran}( \eta) } \big) \ge (\chi/ \delta)^{c_{\star}}  \Big) \le  (\chi/ \delta)^{4}.   
	 \end{equation}
	 Let $\mathsf{R}_\star$ denote the counterpart of the event on the left-hand side of (\ref{show423}) obtained by replacing $\eta$ with $\ell_1$. Since $\mathrm{ran}(\eta)\subset \mathrm{ran}(\ell_1)$, it follows from (\ref{show423}) that $\mathbb{P}\big( \mathsf{R}_\star \big) \le    (\chi/ \delta)^{4}$. For any $x\in \mathbb{R}^3$ and $r>0$, we denote $\mathcal{B}_x(r):=\bm{\mathcal{B}}_x(r)\cap \mathbb{Z}^3$. By the spatial Markov property of loops in $\mathcal{L}_{1/2}^{\mathbb{Z}^d}$, the conditional probability of $\{\mathrm{ran}(\ell_1)\cap \mathrm{ran}(\ell_2)= \emptyset\}$ given $\ell_1$ is bounded from above by 
	 \begin{equation}\label{great424}
	 	\begin{split}
	 	\max_{y_1\in \Lambda_{N\cdot  z}(\chi N)\cap  \mathbb{Z}^3, y_2\in [\mathcal{B}_{N\cdot z}(\frac{\delta N}{4})]^c  }	 & \mathbb{P}_{ y_1}^{\mathbb{Z}^3}\big( \tau_{\bm{\mathcal{C}}_{N\cdot z}(\frac{\delta N}{4})} <\tau_{ N\cdot \mathrm{ran}(\ell_1) }    \big)  \\
	 	 & \cdot \mathbb{P}_{ y_2}^{\mathbb{Z}^3}\big( \tau_{\Lambda_{N\cdot  z}(\chi N)} <\tau_{ N\cdot \mathrm{ran}(\ell_1) }  \big).  
	 	\end{split}
	 \end{equation}
	 Note that when $\mathsf{R}_\star^c$ occurs, one has 
	 \begin{equation}\label{new425}
	 	\mathbb{P}_{ y_1}^{\mathbb{Z}^3}\big( \tau_{\bm{\mathcal{C}}_{N\cdot z}(\frac{\delta N}{4})} <\tau_{ N\cdot \mathrm{ran}(\ell_1) }    \big)  \le (\chi/ \delta)^{c_{\star}}. 
	 \end{equation}
	 Meanwhile, for any $y_2\in [\mathcal{B}_{N\cdot z}(\frac{\delta N}{4})]^c$, by the last-exit decomposition (see e.g., \cite[Proposition 4.6.4]{lawler2010random}), we have 
	 	 \begin{equation}\label{new426}
	 	 	\begin{split}
	 	 		& \mathbb{P}_{ y_2}^{\mathbb{Z}^3}\big( \tau_{\Lambda_{N\cdot  z}(\chi N)} <\tau_{ N\cdot \mathrm{ran}(\ell_1) }  \big) \\
	 	 		\lesssim  & \sum\nolimits_{y_3\in  \partial \mathcal{B}_{N\cdot z}(10 \chi N)} G^{\mathbb{Z}^3}_{N\cdot \mathrm{ran}(\ell_1)}(y_2,y_3) \\ 
	 		   &\cdot \sum\nolimits_{y_3'\in \mathcal{B}_{N\cdot z}(10 \chi N):  \{y_3,y_3'\}\in \mathbb{L}^3 } \mathbb{P}_{ y_3'}^{\mathbb{Z}^3}\big( \tau_{\Lambda_{N\cdot  z}(\chi N)} <\tau_{ \partial \mathcal{B}_{N\cdot z}(10 \chi N) }  \big)\\
	 		   \lesssim  &  (\chi N)^{-1}\sum\nolimits_{y_3\in  \partial \mathcal{B}_{N\cdot z}(10 \chi N)} G^{\mathbb{Z}^3}_{N\cdot \mathrm{ran}(\ell_1)}(y_2,y_3), 
	 	 	\end{split}
	 	 \end{equation}
	 	 where the last inequality follows from \cite[Lemma 6.3.4]{lawler2010random}. Moreover, for any $y_2\in [\mathcal{B}_{N\cdot z}(\frac{\delta N}{4})]^c$ and  $y_3\in  \partial \mathcal{B}_{N\cdot z}(10 \chi N)$, we have 
	 	 \begin{equation}
	 	 	\begin{split}
	 	 		G^{\mathbb{Z}^3}_{N\cdot \mathrm{ran}(\ell_1)}(y_2,y_3) \lesssim & \mathbb{P}_{ y_3}^{\mathbb{Z}^3}\big( \tau_{y_2} <\tau_{ N\cdot \mathrm{ran}(\ell_1) }  \big)  \\
	 	 		\lesssim & \mathbb{P}_{ y_3}^{\mathbb{Z}^3}\big( \tau_{\partial \mathcal{B}_{N\cdot z}(\frac{\delta N}{8})} <\tau_{ N\cdot \mathrm{ran}(\ell_1) }  \big)  \max_{ y_4 \in \partial \mathcal{B}_{N\cdot z}(\frac{\delta N}{8}) }  \mathbb{P}^{\mathbb{Z}^3}_{y_4}(\tau_{y_2}<\infty) \\
	 	 		\lesssim & (\delta N)^{-1}\cdot \mathbb{P}_{ y_3}^{\mathbb{Z}^3}\big( \tau_{\partial \mathcal{B}_{N\cdot z}(\frac{\delta N}{8})} <\tau_{ N\cdot \mathrm{ran}(\ell_1) }  \big). 
	 	 	\end{split}
	 	 \end{equation}
	 	Combined with (\ref{new426}), it yields that 
	 	\begin{equation}\label{new428}
	 	\begin{split}
	 	&	\mathbb{P}_{ y_2}^{\mathbb{Z}^3}\big( \tau_{\Lambda_{N\cdot  z}(\chi N)} <\tau_{ N\cdot \mathrm{ran}(\ell_1) }  \big) \\
	 		 \lesssim & \delta^{-1}  \chi^{-1}N^{-2}\cdot \sum\nolimits_{y_3\in  \partial \mathcal{B}_{N\cdot z}(10 \chi N)}\mathbb{P}_{ y_3}^{\mathbb{Z}^3}\big( \tau_{\partial \mathcal{B}_{N\cdot z}(\frac{\delta N}{8})} <\tau_{ N\cdot \mathrm{ran}(\ell_1) }  \big). 
	 	\end{split}	 		
	 	\end{equation}
	 	 Plugging (\ref{new425}) and (\ref{new428}) into (\ref{great424}), we obtain 	 	    \begin{equation}\label{for429}
	   \begin{split}
	&    \mathbb{P}\big( \widehat{\mathsf{H}},  \mathsf{R}_\star^c \big) = 	  \mathbb{P}\big(|\mathfrak{L}_z|= 2, \mathrm{ran}(\ell_1)\cap \mathrm{ran}(\ell_2)= \emptyset,  \mathsf{R}_\star^c \big) \\
	 	 \lesssim  & \delta^{-1-c_{\star}}\chi^{-1+c_{\star}}N^{-2} \sum_{y_3\in  \partial \mathcal{B}_{N\cdot z}(10 \chi N)}  \mathbb{E}\big[ \mathbbm{1}_{  |\mathfrak{L}_z|\ge 1 } \cdot \mathbb{P}_{ y_3}^{\mathbb{Z}^3}\big( \tau_{\partial \mathcal{B}_{N\cdot z}(\frac{\delta N}{8})} <\tau_{ N\cdot \mathrm{ran}(\ell_1) }  \big)  \big].
	 	 	   \end{split} 
	 \end{equation}
	It was shown in \cite[Remark 6.5]{inpreparation_pivotal} that for any $y_3\in  \partial \mathcal{B}_{N\cdot z}(10 \chi N)$,   	
	 \begin{equation}\label{for430}
	 	 	\mathbb{E}\big[ \mathbbm{1}_{  |\mathfrak{L}_z|\ge 1 } \cdot \mathbb{P}_{ y_3}^{\mathbb{Z}^3}\big( \tau_{\partial \mathcal{B}_{N\cdot z}(\frac{\delta N}{8})} <\tau_{ N\cdot \mathrm{ran}(\ell_1) }  \big)  \big] \lesssim (\chi/ \delta)^{2-\frac{1}{2}c_{\star}}. 
	 	 \end{equation}
	Inserting (\ref{for430}) into (\ref{for429}), and using $| \partial \mathcal{B}_{N\cdot z}(10 \chi N)|\asymp (\chi N)^2$, we derive
	 	 \begin{equation}
	 	 	   \mathbb{P}\big( \widehat{\mathsf{H}} ,  \mathsf{R}_\star^c \big)  \lesssim (\chi/ \delta)^{3+\frac{1}{2}c_{\star}}. 
	 	 \end{equation}
	 	  Combined with $\mathbb{P}\big( \mathsf{R}_\star \big) \lesssim   (\chi/ \delta)^{4}$, it implies that 
	 	 \begin{equation}\label{newadd431}
	 	 	\mathbb{P}\big(\widehat{\mathsf{H}} \big) \lesssim (\chi/ \delta)^{3+\frac{1}{2}c_{\star}}.
	 	 \end{equation}

	 	Substituting (\ref{newadd431}) into (\ref{fix_424}), we complete the proof. 
	  	 	 \end{proof}



 \subsection{Upper bound on the cluster dimension} \label{section_cluster_dimension}

   We arbitrarily fix $x\in \mathbb{R}^3$ and $r_1,r_2>0$. For any $\delta\in (0,\frac{r_1}{100})$, we define 
   \begin{equation}\label{good412}
   	 \Phi_\delta:= \big\{ z\in (2\delta)\cdot \mathbb{Z}^3: \Lambda_z(\delta)\cap \bm{\mathcal{B}}_x(r_1)\neq \emptyset \big\}. 
   \end{equation}
   For any $A\subset \mathbb{R}^3$, its upper box-counting dimension is defined by 
   \begin{equation}\label{good413}
   \overline{\mathrm{dim}}_{\mathrm{B}}(A ):= \limsup_{\delta \downarrow 0}  \frac{\mathrm{log}_{10}\big( | \{z\in (2\delta)\cdot \mathbb{Z}^3: \Lambda_z(\delta)\cap  A\neq \emptyset \}|\big)}{\mathrm{log}_{10}\big(1/\delta \big) }.
   \end{equation}
   As a result, for any $a>0$, if there exists a cluster $\mathcal{C}\in \mathfrak{C}_{r_2}$ (see the definition of $\mathfrak{C}_{r_2}$ before Theorem \ref{theorem_main}) such that $\overline{\mathrm{dim}}_{\mathrm{B}}(\mathcal{C}\cap \bm{\mathcal{B}}_x(r_1))\ge  3-\zeta+a$, then the event 
    \begin{equation}\label{good414}
 \mathsf{G}_a:= \Big\{    \limsup_{\delta \downarrow 0}  \frac{\mathrm{log}_{10}\big( |\Psi_\delta|\big)}{\mathrm{log}_{10}\big(1/\delta \big) }\ge  3-\zeta+a \Big\}  
   \end{equation}  
 occurs, where $\Psi_\delta :=  \{ z\in  \Phi_\delta : \exists \mathcal{C}\in \mathfrak{C}_{r_2}\ \text{such that}\  \Lambda_z(\delta) \cap \mathcal{C}\neq \emptyset  \}$.

  Note that for any $k\in \mathbb{N}$, $\delta\in [10^{-k-1},10^{-k}]$ and $z\in \Phi_\delta$, there exists $z'\in \Phi_{10^{-k+1}}$ such that $\Lambda_z(\delta)\subset \Lambda_{z'}(10^{-k+1})$, and hence any cluster intersecting $\Lambda_z(\delta)$ must hit $\Lambda_{z'}(10^{-k+1})$. Moreover, for each $z'\in \Phi_{10^{-k+1}}$, the box $\Lambda_{z'}(10^{-k+1})$ contains at most $10^6$ boxes in $\{\Lambda_z(\delta): z\in \Phi_\delta\}$. As a result, we have $|\Psi_\delta|\le 10^6\cdot |\Psi_{10^{-k+1}}|$ for all $\delta\in [10^{-k-1},10^{-k}]$, which further implies  
    \begin{equation}\label{new4.15}
   	\mathsf{G}_a\subset \mathsf{G}_a':= \Big\{    \limsup_{k\to \infty}  \frac{\mathrm{log}_{10}\big( |\Psi_{10^{-k+1}}|\big)}{k}\ge  3-\zeta+a \Big\}. 
   \end{equation} 
   On the event $\mathsf{G}_a'$, for each integer $m\ge 1$, there exists $k \ge m$ such that $|\Psi_{10^{-k+1}}|\ge 10^{(3-\zeta+\frac{1}{2}a)k}$ occurs. Thus, for any $m\ge 1$, by Markov's inequality we have 
   \begin{equation}\label{new416}
   \begin{split}
   	 	\mathbb{P}\big(\mathsf{G}_a'  \big) \le & \sum\nolimits_{k\ge m}  \mathbb{P}\big(|\Psi_{10^{-k+1}}|\ge 10^{(3-\zeta+\frac{1}{2}a)k}  \big) \\
   	 	\le &\sum\nolimits_{k\ge m} 10^{-(3-\zeta+\frac{1}{2}a)k} \cdot \mathbb{E}\big[ |\Psi_{10^{-k+1}}|\big] .
   \end{split}
   \end{equation}
For all sufficiently large $k\ge 1$, since $\{ z\in  \Psi_{10^{-k+1}} \} \subset \{ \bm{\mathcal{C}}_{z}(10^{-k+2}) \xleftrightarrow{\cup \mathcal{L}_{1/2}^{\mathbb{R}^3}}  \bm{\mathcal{C}}_z(\frac{1}{2}r_2)   \} $,   
 \begin{equation}\label{new417}
 	\begin{split}
 		\mathbb{E}\big[ |\Psi_{10^{-k+1}}|\big]\le& \big|    \Phi_{10^{-k+1}} \big|\cdot  \mathbb{P}\big(  \bm{\mathcal{C}}(10^{-k+2}) \xleftrightarrow{\cup \mathcal{L}_{1/2}^{\mathbb{R}^3}}  \bm{\mathcal{C}}(\tfrac{1}{2}r_2)  \big)\\
 		\lesssim & 10^{3k} \cdot  \big( \frac{\tfrac{1}{2}r_2 }{10^{-k+2}} \big)^{-\zeta+\frac{1}{4}a}  \asymp  10^{(3-\zeta+\frac{1}{4}a)k}, 
 	\end{split}
 \end{equation}
 where in the second inequality we applied (\ref{new41}) and the existence of the one-arm exponent $\zeta$. Plugging (\ref{new417}) into (\ref{new416}), and then combining with (\ref{new4.15}), we obtain 
 \begin{equation}\label{new418}
 	\mathbb{P}\big(\mathsf{G}_a  \big) \lesssim  \sum\nolimits_{k\ge m}  10^{-\frac{1}{4}ak} =\big[1-10^{-\frac{a}{4}} \big]^{-1}\cdot 10^{-\frac{1}{4}am}. 
 \end{equation}
 Since (\ref{new418}) holds for all $m\ge 1$, one has $\mathbb{P}\big(\mathsf{G}_a  \big)=0$, which further implies
 \begin{equation}
 	\mathbb{P}\big( \exists \mathcal{C}\in \mathfrak{C}_{r_2} \ \text{such that}\   \overline{\mathrm{dim}}_{\mathrm{B}}(\mathcal{C}\cap \bm{\mathcal{B}}_x(r_1))\ge  3-\zeta+a\big) =0.
 \end{equation}
 Noting that $a>0$ is arbitrary, this yields (\ref{ineq_main_2}).  \qed

 \vspace*{0.2cm}

  Recall that $\zeta\le 1$ has been derived from (\ref{newadd210}). To sum up, we have completed the proof of Theorem \ref{theorem_main}. \qed

 \section{One-arm exponent and dimension of the scaling limit}  \label{section5_S_star}

  The goal of this section is to establish Theorem \ref{theorem_main_part2}. We begin by proving the estimate (\ref{ineq_main_part2_1}) for the one-arm probability. For any $R,r>0$, we define the event 
  \begin{equation}
  	\mathsf{F}_{r,R}:=\{  \exists   \mathcal{C}^* \in \mathfrak{C}_{R}^*\ \text{such that}\   \mathcal{C}^* \cap \Lambda(r ) \neq \emptyset  \}. 
  \end{equation}
 Note that there exists $c_\star>0$ such that for any $r>0$, 
 \begin{equation}\label{newfix_box_contain}
 	\Lambda(c_\star r ) \subset \bm{\mathcal{B}}(r)\subset \Lambda( r ). 
 \end{equation}
  Therefore, for all sufficiently small $\epsilon>0$, 
    \begin{equation}\label{new51}
\mathbb{P}\big( \mathsf{F}_{c_\star\epsilon ,2}  \big)  \le  	 \mathbb{P}\big( \bm{\mathcal{C}}(\epsilon) \xleftrightarrow{  \mathcal{S}^*} \bm{\mathcal{C}}(1) \big)  \le  \mathbb{P}\big( \mathsf{F}_{\epsilon ,\frac{1}{2}}  \big).  
  \end{equation}
  For any $n\ge 1$, we denote by $\mathsf{F}_{r,R}^n$ the counterpart of $\mathsf{F}_{r,R}$ obtained by replacing $\mathcal{S}^*$ with $n^{-1}\cdot \cup  \mathcal{L}_{1/2}^{\widetilde{\mathbb{Z}}^d}$. It follows from Property (II) presented before Theorem \ref{theorem_main_part2} that  
  \begin{equation}\label{new52}
  	\mathbb{P}\big(\mathsf{F}_{r,R}  \big)= \lim\limits_{n\to \infty}\mathbb{P}\big(\mathsf{F}_{r,R}^n  \big), \ \forall R,r>0.
  \end{equation}

  Recall $\rho_d$ in (\ref{def_rho_crossing}). For any $r>0$ and $R>200c_\star^{-1}r$, the inclusion (\ref{newfix_box_contain}) implies
  \begin{equation}\label{new53}
 \rho_d(\tfrac{rn}{10}, 20Rn)     \le 	\mathbb{P}\big(\mathsf{F}_{r,R}^n  \big) \le   \rho_d(10rn, \tfrac{c_\star Rn}{20}). 
  \end{equation}
  Moreover, the estimate (\ref{crossing_low}) implies  
  \begin{equation}\label{new54}
\rho_d(\tfrac{rn}{10}, 20Rn)   \asymp \rho_d(10rn, \tfrac{c_\star Rn}{20}) \asymp \big( r/R \big)^{\frac{d}{2}-1}.  
  \end{equation} 
    Putting (\ref{new52}), (\ref{new53}) and (\ref{new54}) together, we get 
  \begin{equation}
  	\mathbb{P}\big(\mathsf{F}_{r,R}  \big) \asymp \big( r/R \big)^{\frac{d}{2}-1}.  \end{equation}
  Combined with (\ref{new51}), it yields Item (1) of Theorem \ref{theorem_main_part2}.


  We now turn to the proof of Item (2). Arbitrarily fix $x\in \mathbb{R}^d$ and $r_1,r_2>0$. By repeating the arguments in Section \ref{section_cluster_dimension} (with the one-arm estimate (\ref{ineq_main1}) replaced by (\ref{ineq_main_part2_1})), one has   
   \begin{equation}\label{new56}
   	\mathbb{P}\big( \overline{\mathrm{dim}}_{\mathrm{B}}(\mathcal{C}^*\cap \bm{\mathcal{B}}_x(r_1))\le \tfrac{d}{2}+1, \forall   \mathcal{C}^*\in \mathfrak{C}_{r_2}^*\big)  =1. 
   \end{equation}
   Note that the dimension bound $\tfrac{d}{2}+1$ in (\ref{new56}) equals $d-(\tfrac{d}{2}-1)$. Here $d$ and $\tfrac{d}{2}-1$ are the Euclidean dimension and the one-arm exponent, playing the roles of $3$ and $\zeta$ in (\ref{ineq_main_2}) respectively. By (\ref{new56}), it is sufficient to show that 
   \begin{equation}\label{new57}
   \mathbb{P}(\mathsf{G}):=  \mathbb{P}\big( \exists \mathcal{C}^* \in \mathfrak{C}_{r_2}^*\ \text{such that}\ 
 \overline{\mathrm{dim}}_{\mathrm{B}}( \mathcal{C}^* \cap \bm{\mathcal{B}}_x(r_1)) \ge \tfrac{d}{2}+1\big) >0.  
   \end{equation}
   To this end, we take a large constant $C_{\clubsuit}>0$ and define 
   \begin{equation}
   		\Phi_\delta :=  \big\{ z\in (2\delta)\cdot \mathbb{Z}^d: \Lambda_z(\delta)\cap \bm{\mathcal{B}}_x( \tfrac{r_1\land r_2}{C_{\clubsuit}}   )\neq \emptyset \big\}, 
   \end{equation}
   \begin{equation}
   	 \Psi_\delta^*:=   \{ z\in  \Phi_\delta : \exists \mathcal{C}^*\in \mathfrak{C}_{r_2}^*\ \text{such that} \   \Lambda_z(\delta) \cap \mathcal{C}^*\neq \emptyset  \}. 
   \end{equation}
For any $k\in \mathbb{N}^+$ and $a>0$, we define the event $\mathsf{H}_{k,a}:= \big\{   | \Psi_{10^{-k}}^* |  \ge a\cdot 10^{(\frac{d}{2}+1)k}  \big\}$. Let $\widehat{\mathfrak{C}}_{r_2}^{*}$ denote the collection of clusters in $\mathfrak{C}_{r_2}^*$ that intersect $\bm{\mathcal{B}}_x(2r_1)$.

We claim the following results:
 \begin{itemize}

 	\item There exist constants $c_{\spadesuit},c_{\diamondsuit}>0$ such that 
   \begin{equation}\label{claim59}
   	\mathbb{P}(\mathsf{H}):= \mathbb{P}\big( \cap_{m\ge 1}\cup_{k\ge m} \mathsf{H}_{k,c_{\spadesuit}} \big) \ge  c_{\diamondsuit}. 
   \end{equation}

 	\item There exists a constant $C_{\heartsuit}>0$ such that 
 	 \begin{equation}\label{claim510}
 	 	\mathbb{P}\big( |\widehat{\mathfrak{C}}_{r_2}^* | \ge C_{\heartsuit}  \big)  \le \tfrac{1}{2}c_{\diamondsuit}. 
 	 \end{equation}

 \end{itemize}
We first prove (\ref{new57}) assuming them. By (\ref{claim59}) and (\ref{claim510}), one has 
\begin{equation}\label{newadd511}
	\mathbb{P}\big(\mathsf{H} \cap  \{|\widehat{\mathfrak{C}}_{r_2}^* | <  C_{\heartsuit}   \} \big)  \ge \tfrac{1}{2}c_{\diamondsuit}. 
\end{equation}
On $\{|\widehat{\mathfrak{C}}_{r_2}^* | <  C_{\heartsuit}\}$, we enumerate all clusters in $\widehat{\mathfrak{C}}_{r_2}^*$ as $\{ \mathcal{C}^*_j\}_{1\le j\le K}$, where $K < C_{\heartsuit}$. Note that for all sufficiently large $k$, every cluster contributing to $\Psi_{10^{-k}}^*$ must intersect $\bm{\mathcal{B}}_x(2r_1)$, and hence belongs to $\widehat{\mathfrak{C}}_{r_2}^*$. Therefore, when $\mathsf{H}_{k,c_{\spadesuit}} \cap  \{|\widehat{\mathfrak{C}}_{r_2}^* | <  C_{\heartsuit}   \}$ occurs, by the pigeonhole principle there exists $j\in [1,K]$ such that $\mathcal{C}^*_j$ intersects at least $(C_{\heartsuit})^{-1}c_{\spadesuit}10^{(\frac{d}{2}+1)k}$ boxes in $\{\Lambda_z(10^{-k}): z\in \Phi_{10^{-k}}\}$ (we denote this event by $\mathsf{H}_{k,c_{\spadesuit}}^{(j)}$). Consequently, on $\mathsf{H} \cap  \{|\widehat{\mathfrak{C}}_{r_2}^* | <  C_{\heartsuit}\}$ (where $\{\mathsf{H}_{k,c_{\spadesuit}}\}_{k\ge 1}$ occur infinitely often), there must exist $j_\dagger\in [1,K]$ such that $\{\mathsf{H}_{k,c_{\spadesuit}}^{(j_\dagger)}\}_{k\ge 1}$ occur infinitely often. This together with the formula (\ref{good413}) yields that $\overline{\mathrm{dim}}_{\mathrm{B}}( \mathcal{C}^*_{j_\dagger}\cap \bm{\mathcal{B}}_x(r_1)) \ge \tfrac{d}{2}+1$, and hence $\mathsf{G}$ occurs. To sum up, by (\ref{newadd511}) we obtain the desired bound (\ref{new57}): 
\begin{equation}
	 \mathbb{P}(\mathsf{G}) \ge \mathbb{P}\big(\mathsf{H} \cap  \{|\widehat{\mathfrak{C}}_{r_2}^* | <  C_{\heartsuit}   \} \big)  \ge  \tfrac{1}{2}c_{\diamondsuit}. 
\end{equation}

 In what follows, we present the proofs of (\ref{claim59}) and (\ref{claim510}).

 \textbf{Proof of (\ref{claim59}).} Applying Fatou's Lemma, one has \
    \begin{equation}\label{5.10}
    	\begin{split}
    		\mathbb{P}(\mathsf{H}) =	\mathbb{E}\big[ \limsup_{k\to \infty} \mathbbm{1}_{\mathsf{H}_{k,c_{\spadesuit}}} \big]  = &  1- \mathbb{E}\big[ \liminf_{k\to \infty} \mathbbm{1}_{(\mathsf{H}_{k,c_{\spadesuit}})^c} \big]   \\
 \ge &1-  \liminf_{k\to \infty} \mathbb{P}\big( (\mathsf{H}_{k,c_{\spadesuit}})^c \big) =\limsup_{k\to \infty}  \mathbb{P}\big( \mathsf{H}_{k,c_{\spadesuit}} \big). 
    	\end{split}
    \end{equation}
     We denote by $\mathsf{H}_{k,c_{\spadesuit}}^n$ the analogue of $\mathsf{H}_{k,c_{\spadesuit}}$ obtained by replacing $\mathcal{S}^*$ with $n^{-1}\cdot \cup  \mathcal{L}_{1/2}^{\widetilde{\mathbb{Z}}^d}$. As in (\ref{new52}), Property (II) of the scaling limit $\mathcal{S}^*$ implies $\mathbb{P}\big( \mathsf{H}_{k,c_{\spadesuit}} \big)= \lim\limits_{n\to \infty}\mathbb{P}\big( \mathsf{H}_{k,c_{\spadesuit}}^n \big)$. 
  Combined with (\ref{5.10}), it yields 
  \begin{equation}\label{newnew512}
  	\mathbb{P}(\mathsf{H})\ge \limsup_{k\to \infty} \big[  \lim\limits_{n\to \infty}\mathbb{P}\big( \mathsf{H}_{k,c_{\spadesuit}}^n \big)\big].  
  \end{equation}

 Next, we apply the second moment method to show that $\mathbb{P}\big( \mathsf{H}_{k,c_{\spadesuit}}^n \big)$ is uniformly bounded away from zero. For any $k,n\ge 1$ and $z\in  \Phi_{10^{-k}}$, we define $\mathsf{J}^z_{k,n}$ as the event that there exists a cluster in $n^{-1}\cdot \cup  \mathcal{L}_{1/2}^{\widetilde{\mathbb{Z}}^d}$ that has diameter at least $r_2$ and intersects the box $\Lambda_z(10^{-k})$. Let $\mathbb{J}_{_{k,n}}:=\sum_{z\in \Phi_{10^{-k}}} \mathbbm{1}_{\mathsf{J}^z_{k,n}}$. Note that $\mathsf{H}_{k,c_{\spadesuit}}^n= \{\mathbb{J}_{_{k,n}} \ge c_{\spadesuit}10^{(\frac{d}{2}+1)k} \}$, and that for all sufficiently large $k$, 
  \begin{equation}\label{new513}
\big\{ \Lambda_z(10^{-k}) \xleftrightarrow{n^{-1}\cdot \cup  \mathcal{L}_{1/2}^{\widetilde{\mathbb{Z}}^d} } \bm{\mathcal{C}}_z(2r_2)     \big\}  	\subset   \mathsf{J}^z_{k,n} \subset  \big\{ \Lambda_z(10^{-k}) \xleftrightarrow{n^{-1}\cdot \cup  \mathcal{L}_{1/2}^{\widetilde{\mathbb{Z}}^d} } \bm{\mathcal{C}}_z(\tfrac{1}{2}r_2)     \big\}.  
 \end{equation} 
 Thus, for the first moment of $\mathbb{J}_{_{k,n}}$, we have 
 \begin{equation}\label{new514}
 	\begin{split}
 		\mathbb{E}\big[\mathbb{J}_{_{k,n}} \big] = \sum\nolimits_{z\in \Phi_{10^{-k}}} \mathbb{P}\big( \mathsf{J}^z_{k,n}\big)  \ge \big| \Phi_{10^{-k}} \big|\cdot \rho_d(10^{-k-1}n, 20r_2n ) \overset{(\ref{crossing_low})}{\gtrsim }10^{(\frac{d}{2}+1)k} . 
 	\end{split}
 \end{equation}
 Meanwhile, the second moment of $\mathbb{J}_{_{k,n}}$ can be written as 
 \begin{equation}\label{new515}
 	\begin{split}
 		\mathbb{E}\big[\mathbb{J}_{_{k,n}}^2 \big]  =&   \sum\nolimits_{z_1,z_2\in \Phi_{10^{-k}}} \mathbb{P}\big( \mathsf{J}^{z_1}_{k,n}\cap \mathsf{J}^{z_2}_{k,n} \big).
 		 	\end{split}
 \end{equation}
 In what follows, we estimate the probability of $\mathsf{J}^{z_1}_{k,n}\cap \mathsf{J}^{z_2}_{k,n}$ separately for the cases $\|z_1-z_2\|\le C_{\clubsuit}\cdot 10^{-k}$ and $\|z_1-z_2\|> C_{\clubsuit}\cdot 10^{-k}$.




 \textbf{When $\|z_1-z_2\|\le C_{\clubsuit}\cdot 10^{-k}$.} For any $z\in \Phi_{10^{-k}}$, we have 
 \begin{equation}\label{new516}
 	\begin{split}
 		   \mathbb{P}\big( \mathsf{J}^{z}_{k,n} \big) \overset{(\ref{newfix_box_contain}),(\ref{new513})}{\le} \rho_d(10^{-k+1}n, \tfrac{1}{20}c_\star r_2n )   \overset{(\ref{crossing_low})}{\lesssim  } 10^{(-\frac{d}{2}+1)k}. 
 	\end{split}
 \end{equation}
It then follows from (\ref{new516}) that  
 \begin{equation}\label{new517}
 	\begin{split}
 		\mathbb{P}\big( \mathsf{J}^{z_1}_{k,n}\cap \mathsf{J}^{z_2}_{k,n} \big) \overset{}{\le}  \mathbb{P}\big( \mathsf{J}^{z_1}_{k,n} \big)   \overset{ }{\lesssim  } 10^{(-\frac{d}{2}+1)k}. 
 	\end{split}
 \end{equation}

  \textbf{When $\|z_1-z_2\|> C_{\clubsuit}\cdot 10^{-k}$.} Let $\mathsf{K}^{(1)}_{z_1,z_2}$ denote the event that there exist two distinct clusters in $n^{-1}\cdot \cup  \mathcal{L}_{1/2}^{\widetilde{\mathbb{Z}}^d}$ certifying $\mathsf{J}^{z_1}_{k,n}$ and $\mathsf{J}^{z_2}_{k,n}$ respectively. We also denote $\mathsf{K}^{(2)}_{z_1,z_2}:=\mathsf{J}^{z_1}_{k,n}\cap \mathsf{J}^{z_2}_{k,n} \cap [ \mathsf{K}^{(1)}_{z_1,z_2}]^c$. By the BKR inequality, we have 
 \begin{equation}\label{new518}
 	\mathbb{P}\big( \mathsf{K}^{(1)}_{z_1,z_2}\big) \le \mathbb{P}\big( \mathsf{J}^{z_1}_{k,n} \big) \cdot \mathbb{P}\big(   \mathsf{J}^{z_2}_{k,n} \big)  \overset{(\ref{new516})}{\lesssim } 10^{(2-d)k}.
 \end{equation}
On $\mathsf{K}^{(2)}_{z_1,z_2}$, there exists a cluster in $n^{-1}\cdot \cup  \mathcal{L}_{1/2}^{\widetilde{\mathbb{Z}}^d}$ intersecting $\Lambda_{z_1}(10^{-k})$, $\Lambda_{z_2}(10^{-k})$ and $\bm{\mathcal{C}}_{z_1}(\frac{1}{2}r_2)$ (we denote this event by $\overline{\mathsf{K}}^{(2)}_{z_1,z_2}$). It follows from \cite[Lemma 5.2]{cai2024quasi} that 
 \begin{equation}\label{new519}
 	\begin{split}
 		\mathbb{P}\big( \overline{\mathsf{ K}}^{(2)}_{z_1,z_2}\big) \lesssim &  \mathbb{P}\big( \Lambda_{z_1}(10^{-k}) \xleftrightarrow{n^{-1}\cdot \cup  \mathcal{L}_{1/2}^{\widetilde{\mathbb{Z}}^d}}  \bm{\mathcal{C}}_{z_1}(C_{\clubsuit}^{1/4}\|z_1-z_2\|)  \big) \\
 		&\cdot  \Big[ \mathbb{P}\big( \Lambda_{z_2}(10^{-k}) \xleftrightarrow{n^{-1}\cdot \cup  \mathcal{L}_{1/2}^{\widetilde{\mathbb{Z}}^d}}  \bm{\mathcal{C}}_{z_1}(\tfrac{1}{2}r_2)  \big)  \\
 		&\ \ \ \ +\mathbb{P}\big( \Lambda_{z_2}(10^{-k}) \xleftrightarrow{n^{-1}\cdot \cup  \mathcal{L}_{1/2}^{\widetilde{\mathbb{Z}}^d}}  \bm{\mathcal{C}}_{z_1}(C_{\clubsuit}^{1/2}\|z_1-z_2\|)  \big)\\
 		&\ \ \ \ \ \ \ \cdot \mathbb{P}\big(  \bm{\mathcal{C}}_{z_1}(\tfrac{1}{2}r_2)\xleftrightarrow{n^{-1}\cdot \cup  \mathcal{L}_{1/2}^{\widetilde{\mathbb{Z}}^d}}  \bm{\mathcal{C}}_{z_1}(C_{\clubsuit}^{1/2}\|z_1-z_2\|)  \big)\Big] \\
 		\overset{(\ref{crossing_low})}{\lesssim } & \big( \tfrac{10^{-k}}{\|z_1-z_2\|} \big)^{\frac{d}{2}-1}\cdot   \big( \tfrac{10^{-k}}{r_2} \big)^{\frac{d}{2}-1}  \asymp \|z_1-z_2\|^{-\frac{d}{2}+1}10^{(2-d)k}. 
 	\end{split}
 \end{equation}
Note that $\|z_1-z_2\|\lesssim r_1$. Thus, combining (\ref{new518}) and (\ref{new519}), we get 
 \begin{equation}\label{new520}
 	\begin{split}
 		\mathbb{P}\big( \mathsf{J}^{z_1}_{k,n}\cap \mathsf{J}^{z_2}_{k,n} \big)    \overset{ }{\lesssim  } \|z_1-z_2\|^{-\frac{d}{2}+1}10^{(2-d)k}. 
 	\end{split}
 \end{equation}

By (\ref{new515}), (\ref{new517}) and (\ref{new520}), we have 
\begin{equation}\label{new521}
	\begin{split}
		\mathbb{E}\big[\mathbb{J}_{_{k,n}}^2 \big]  \lesssim   &  \big| \{(z_1,z_2)\in (\Phi_{10^{-k}})^2: \|z_1-z_2\|\le C_{\clubsuit}\cdot 10^{-k}\} \big|\cdot 10^{(-\frac{d}{2}+1)k}\\
		&+\sum\nolimits_{z_1,z_2\in \Phi_{10^{-k}}:\|z_1-z_2\|> C_{\clubsuit}\cdot 10^{-k}}   \|z_1-z_2\|^{-\frac{d}{2}+1}10^{(2-d)k}\\
		\lesssim &   10^{(\frac{d}{2}+1)k}+ 10^{2k}\cdot \max_{z_1\in \Phi_{10^{-k}} }\sum_{z_2\in \Phi_{10^{-k}}:\|z_1-z_2\|> C_{\clubsuit}\cdot 10^{-k}} \|z_1-z_2\|^{-\frac{d}{2}+1}\\
		\lesssim &     10^{(\frac{d}{2}+1)k}\cdot  \Big(1+ \sum\nolimits_{w\in  B(Cr_1\cdot 10^{k} )} (\|w\|+1)^{-\frac{d}{2}+1} \Big) \lesssim 10^{(d+2)k}. 
	\end{split}
\end{equation}
Recall that $\mathsf{H}_{k,c_{\spadesuit}}^n= \{\mathbb{J}_{_{k,n}} \ge c_{\spadesuit}10^{(\frac{d}{2}+1)k} \}$. Thus, by applying the Paley-Zygmund inequality and then invoking (\ref{new514}) and (\ref{new521}), we have 
\begin{equation}\label{new522}
	\mathbb{P}\big(  \mathsf{H}_{k,c_{\spadesuit}}^n  \big) = \mathbb{P}\big(\mathbb{J}_{_{k,n}} \ge c_{\spadesuit}10^{(\frac{d}{2}+1)k}\big)  \gtrsim \big( \mathbb{E}\big[\mathbb{J}_{_{k,n}} \big]\big)^2 /  \mathbb{E}\big[\mathbb{J}_{_{k,n}}^2 \big] \gtrsim 1. 
\end{equation}

Plugging (\ref{new522}) into (\ref{newnew512}), we derive the claim (\ref{claim59}).  \qed

\vspace{0.2cm}

\textbf{Proof of (\ref{claim510}).} Let $\mathfrak{C}^n_{r_2}$ denote the collection of clusters in $n^{-1}\cdot \cup  \mathcal{L}_{1/2}^{\widetilde{\mathbb{Z}}^d}$ with diameter at least $r_2$. It follows from (\ref{crossing_low}) that for any $k\in \mathbb{N}^+$ and $z\in \mathbb{R}^d$,
\begin{equation}\label{improve5.28}
	\mathbb{P}\big(  \exists \mathcal{C} \in \mathfrak{C}^n_{r_2}  \ \text{such that}\     \Lambda_z(10^{-k})\cap \mathcal{C}  \neq \emptyset \big) \lesssim \big(10^{k} r_2 \big)^{-\frac{d}{2}+1} 
\end{equation}    
holds uniformly for all $n\ge 1$. By the BKR inequality and (\ref{improve5.28}), there exists a sufficiently large integer $k_\dagger(d,r_2)  \ge 1$ such that for any $l\in \mathbb{N}^+$,  
\begin{equation}
	\begin{split}
		&   \mathbb{P}\big( |\{\mathcal{C} \in  \mathfrak{C}^n_{r_2}:  \Lambda_z(10^{-k_\dagger})\cap \mathcal{C}  \neq \emptyset \}|\ge l  \big) \\
		   \le &\big[  \mathbb{P}\big(  \exists \mathcal{C} \in \mathfrak{C}^n_{r_2}  \ \text{such that}\     \Lambda_z(10^{-k_\dagger})\cap \mathcal{C}  \neq \emptyset \big)  \big]^l  \le   2^{-l}.
	\end{split}
\end{equation} 
Letting $n \to \infty$, we obtain 
\begin{equation}\label{improve530}
 \mathbb{P}\big( \mathsf{R}_{z,l}^*  \big):= \mathbb{P}\big( |\{\mathcal{C}^* \in  \mathfrak{C}^*_{r_2}:  \Lambda_z(10^{-k_\dagger})\cap \mathcal{C}  \neq \emptyset \}|\ge l  \big) \le   2^{-l}.
\end{equation}
Let $\Phi_{\dagger}:= \{ z\in  (2\cdot 10^{-k_\dagger})\cdot \mathbb{Z}^d: \Lambda_z(10^{-k_\dagger})\cap  \bm{\mathcal{B}}_x(2r_1) \neq \emptyset \}$. Note that $|\Phi_{\dagger}|\lesssim 10^{k_\dagger d}r_1^d$, and that every cluster in $\widehat{\mathfrak{C}}_{r_2}^*$ must intersect $\Lambda_z(10^{-k_\dagger})$ for some $z\in \Phi_{\dagger}$. Hence, if $|\widehat{\mathfrak{C}}_{r_2}^* | \ge C_{\heartsuit}$, then there exists $z\in \Phi_{\dagger}$ such that $\mathsf{R}_{z,C_{\heartsuit} \cdot |\Phi_{\dagger}|^{-1}}^*$ occurs. Therefore, by the union bound and (\ref{improve530}), one has
 \begin{equation}\label{add528}
	\mathbb{P}\big( |\widehat{\mathfrak{C}}_{r_2}^* | \ge C_{\heartsuit}  \big) \le \sum\nolimits_{z\in \Phi_{\dagger}} \mathbb{P}\big(  \mathsf{R}_{z,C_{\heartsuit}\cdot |\Phi_{\dagger}|^{-1}}^* \big)\le |\Phi_{\dagger}|\cdot 2^{-C_{\heartsuit}\cdot |\Phi_{\dagger}|^{-1}}. 
\end{equation}
 By taking a sufficiently large $C_{\heartsuit}>0$ (depending on $d$, $r_1$ and  $r_2$), we derive the claim (\ref{claim510}) from (\ref{add528}). 
 \qed

 \vspace{0.2cm}

In conclusion, we have verified (\ref{new57}), which together with (\ref{new56}) yields Item (2). This completes the proof of Theorem \ref{theorem_main_part2}.  \qed


 \section*{Acknowledgments}


 We would like to express our sincere gratitude to Wendelin Werner for his insightful remarks, which significantly enhanced the clarity of the manuscript. We warmly thank Yifan Gao for pointing out the reference \cite{sapozhnikov2018brownian}. J. Ding is supported by the National Natural Science Foundation of China (Grant No. 12231002, 12595284, 12595280), and by the New Cornerstone Science Foundation through the New Cornerstone Investigator Program and XPLORER PRIZE.





\begin{thebibliography}{10}

\bibitem{cai2024incipient}
Z.~Cai and J.~Ding.
\newblock {Incipient infinite clusters and self-similarity for metric graph Gaussian free fields and loop soups}.
\newblock {\em arXiv preprint arXiv:2412.05709}, 2024.

\bibitem{cai2024high}
Z.~Cai and J.~Ding.
\newblock {One-arm exponent of critical level-set for metric graph Gaussian free field in high dimensions}.
\newblock {\em Probability Theory and Related Fields}, pages 1--86, 2024.

\bibitem{cai2024one}
Z.~Cai and J.~Ding.
\newblock {One-arm probabilities for metric graph Gaussian free fields below and at the critical dimension}.
\newblock {\em arXiv preprint arXiv:2406.02397}, 2024.

\bibitem{cai2024quasi}
Z.~Cai and J.~Ding.
\newblock {Quasi-multiplicativity and regularity for critical metric graph Gaussian free fields}.
\newblock {\em arXiv preprint arXiv:2412.05706}, 2024.

\bibitem{inpreparation_twoarm}
Z.~Cai and J.~Ding.
\newblock {Heterochromatic two-arm probabilities for metric graph Gaussian free fields}.
\newblock {\em arXiv preprint arXiv:2510.20492}, 2025.

\bibitem{inpreparation_pivotal}
Z.~Cai and J.~Ding.
\newblock {Separation and cut edge in macroscopic clusters for metric graph Gaussian free fields}.
\newblock {\em arXiv preprint arXiv:2510.20516}, 2025.

\bibitem{camia2006two}
F.~Camia and C.~M. Newman.
\newblock Two-dimensional critical percolation: the full scaling limit.
\newblock {\em Communications in Mathematical Physics}, 268(1):1--38, 2006.

\bibitem{chang2017supercritical}
Y.~Chang.
\newblock Supercritical loop percolation on $\mathbb{Z}^d$ for $d\ge 3$.
\newblock {\em Stochastic Processes and their Applications}, 127(10):3159--3186, 2017.

\bibitem{chang2024percolation}
Y.~Chang, H.~Du, and X.~Li.
\newblock Percolation threshold for metric graph loop soup.
\newblock {\em Bernoulli}, 30(4):3324--3333, 2024.

\bibitem{chang2016phase}
Y.~Chang and A.~Sapozhnikov.
\newblock Phase transition in loop percolation.
\newblock {\em Probability Theory and Related Fields}, 164(3):979--1025, 2016.

\bibitem{ding2020percolation}
J.~Ding and M.~Wirth.
\newblock Percolation for level-sets of {Gaussian} free fields on metric graphs.
\newblock {\em The Annals of Probability}, 48(3):1411--1435, 2020.

\bibitem{drewitz2022cluster}
A.~Drewitz, A.~Pr{\'e}vost, and P.-F. Rodriguez.
\newblock {Cluster capacity functionals and isomorphism theorems for Gaussian free fields}.
\newblock {\em Probability Theory and Related Fields}, 183(1):255--313, 2022.

\bibitem{drewitz2023arm}
A.~Drewitz, A.~Pr{\'e}vost, and P.-F. Rodriguez.
\newblock {Arm exponent for the Gaussian free field on metric graphs in intermediate dimensions}.
\newblock {\em arXiv preprint arXiv:2312.10030}, 2023.

\bibitem{drewitz2023critical}
A.~Drewitz, A.~Pr{\'e}vost, and P.-F. Rodriguez.
\newblock Critical exponents for a percolation model on transient graphs.
\newblock {\em Inventiones mathematicae}, 232(1):229--299, 2023.

\bibitem{drewitz2024cluster}
A.~Drewitz, A.~Pr{\'e}vost, and P.-F. Rodriguez.
\newblock {Cluster volumes for the Gaussian free field on metric graphs}.
\newblock {\em arXiv preprint arXiv:2412.06772}, 2024.

\bibitem{drewitz2024critical}
A.~Drewitz, A.~Pr{\'e}vost, and P.-F. Rodriguez.
\newblock {Critical one-arm probability for the metric Gaussian free field in low dimensions}.
\newblock {\em Probability Theory and Related Fields}, pages 1--24, 2025.

\bibitem{falconer2013fractal}
K.~Falconer.
\newblock {\em Fractal geometry: mathematical foundations and applications}.
\newblock John Wiley \& Sons, 2013.

\bibitem{ganguly2024critical}
S.~Ganguly and K.~Jing.
\newblock Critical level set percolation for the {GFF} in $d> 6$: comparison principles and some consequences.
\newblock {\em arXiv preprint arXiv:2412.17768}, 2024.

\bibitem{ganguly2024ant}
S.~Ganguly and K.~Nam.
\newblock {The ant on loops: Alexander-Orbach conjecture for the critical level set of the Gaussian free field}.
\newblock {\em arXiv preprint arXiv:2403.02318}, 2024.

\bibitem{cutpointlemma}
G.~Lawler.
\newblock {Cut times for simple random walk}.
\newblock {\em Electronic Journal of Probability}, 1:1 -- 24, 1996.

\bibitem{lawler2007random}
G.~Lawler and J.~Trujillo~Ferreras.
\newblock Random walk loop soup.
\newblock {\em Transactions of the American Mathematical Society}, 359(2):767--787, 2007.

\bibitem{lawler2010random}
G.~F. Lawler and V.~Limic.
\newblock {\em Random walk: a modern introduction}, volume 123.
\newblock Cambridge University Press, 2010.

\bibitem{le2011markov}
Y.~Le~Jan.
\newblock {\em Markov Paths, Loops and Fields: {\'E}cole d'{\'E}t{\'e} de Probabilit{\'e}s de Saint-Flour XXXVIII--2008}, volume 2026.
\newblock Springer Science \& Business Media, 2011.

\bibitem{lupu2016loop}
T.~Lupu.
\newblock From loop clusters and random interlacements to the free field.
\newblock {\em Annals of Probability}, 44(3):2117--2146, 2016.

\bibitem{lupu2016loop2d}
T.~Lupu.
\newblock Loop percolation on discrete half-plane.
\newblock {\em Electronic Communications in Probability}, 21(30):1--9, 2016.

\bibitem{lupu2019convergence}
T.~Lupu.
\newblock Convergence of the two-dimensional random walk loop soup clusters to {CLE}.
\newblock {\em J. Eur. Math. Soc}, 21(4):1201--1227, 2019.

\bibitem{morters2010brownian}
P.~M{\"o}rters and Y.~Peres.
\newblock {\em Brownian motion}, volume~30.
\newblock Cambridge University Press, 2010.

\bibitem{sapozhnikov2018brownian}
A.~Sapozhnikov and D.~Shiraishi.
\newblock {On Brownian motion, simple paths, and loops}.
\newblock {\em Probability Theory and Related Fields}, 172(3):615--662, 2018.

\bibitem{sheffield2012conformal}
S.~Sheffield and W.~Werner.
\newblock {Conformal loop ensembles: the Markovian characterization and the loop-soup construction}.
\newblock {\em Annals of Mathematics}, pages 1827--1917, 2012.

\bibitem{10.1214/ECP.v17-1792}
A.-S. Sznitman.
\newblock {An isomorphism theorem for random interlacements}.
\newblock {\em Electronic Communications in Probability}, 17:1--9, 2012.

\bibitem{werner2016spatial}
W.~Werner.
\newblock On the spatial {Markov} property of soups of unoriented and oriented loops.
\newblock In {\em S{\'e}minaire de Probabilit{\'e}s XLVIII}, pages 481--503. Springer, 2016.

\bibitem{werner2021clusters}
W.~Werner.
\newblock On clusters of {B}rownian loops in $d$ dimensions.
\newblock {\em In and Out of Equilibrium 3: Celebrating Vladas Sidoravicius}, pages 797--817, 2021.

\bibitem{werner2025switching}
W.~Werner.
\newblock {A switching identity for cable-graph loop soups and Gaussian free fields}.
\newblock {\em arXiv preprint arXiv:2502.06754}, 2025.

\end{thebibliography}

\end{document}